\documentclass[reqno,12pt, centertags]{amsart}
\usepackage[left=1in,right=1in,bottom=1in,top=1in]{geometry}
\usepackage{amsmath,amsthm,amscd,amssymb}
\usepackage{amsmath,amsthm,amscd,amssymb,color,eucal,enumitem,latexsym,lscape,graphicx,multirow,mathrsfs,mathtools}
\usepackage{upref}
\usepackage{latexsym}
\usepackage{lscape}
\usepackage{graphicx}
\usepackage{multirow}
\usepackage{mathrsfs} 

\usepackage{mathtools}

\mathtoolsset{showonlyrefs}

\makeatletter
\def\theequation{\@arabic\c@equation}

\newcommand{\im}{\text{Im}}

\newcommand{\dd}{d}

\newcommand{\bbN}{{\mathbb{N}}}
\newcommand{\bbR}{{\mathbb{R}}}
\newcommand{\bbT}{{\mathbb{T}}}

\newcommand{\bbP}{{\mathbb{P}}}

\newcommand{\bbZ}{{\mathbb{Z}}}

\newcommand{\bbE}{{\mathbb{E}}}
\newcommand{\bbC}{{\mathbb{C}}}

\newcommand{\cB}{{\mathcal B}}

\newcommand{\cD}{{\mathcal D}}
\newcommand{\cE}{{\mathcal E}}
\newcommand{\cF}{{\mathcal F}}

\newcommand{\cI}{{\mathcal I}}

\newcommand{\cK}{{\mathcal K}}

\newcommand{\cM}{{\mathcal M}}

\newcommand{\cO}{{\mathcal O}}

\newcommand{\cU}{{\mathcal U}}

\newcommand{\bfi}{{\bf i}}

\newcommand{\no}{\nonumber}
\newcommand{\lb}{\label}

\newcommand{\ol}{\overline}

\newcommand{\wti}{\widetilde  }

\newcommand{\unif}{\text{unif}}
\newcommand{\1}{[1]}

\newcommand{\rank}{\text{\rm{rank}}}
\newcommand{\Ar}{\text{\rm{Ar}}}

\newcommand{\ac}{\text{\rm{ac}}}
\newcommand{\singc}{\text{\rm{sc}}}
\newcommand{\pp}{\text{\rm{pp}}}

\newcommand{\AC}{\text{\rm{AC}}_{\loc}}

\newcommand{\hatt}{\widehat}

\newcommand{\dist}{\operatorname{dist}}

\newcommand{\SL}{\mathrm{SL}(2,\bbR)}

\numberwithin{equation}{section}

\renewcommand{\det}{\operatorname{det}}
\newcommand{\loc}{\operatorname{loc}}
\newcommand{\dom}{\operatorname{dom}}

\newcommand{\tr}{\operatorname{Tr}}
\newcommand{\re}{\operatorname{Re}}
\newcommand{\supp}{\operatorname{supp}}

\newcommand{\ess}{\operatorname{ess}}

\newcommand{\spec}{\operatorname{Spec}}

\renewcommand{\Re}{\operatorname{Re }}
\renewcommand\Im{\operatorname{Im}}

\newcommand{\Int}{\operatorname{int}}

\newtheorem{theorem}{Theorem}[section]
\newtheorem{hypothesis}[theorem]{Hypothesis}
\newtheorem{lemma}[theorem]{Lemma}
\newtheorem{corollary}[theorem]{Corollary}
\newtheorem{proposition}[theorem]{Proposition}

\theoremstyle{definition}

\newtheorem{example}[theorem]{Example}

\newtheorem{remark}[theorem]{Remark}

\begin{document} 
	
\numberwithin{equation}{section}
\allowdisplaybreaks

\title[Schr\"odinger operators with locally $H^{-1}$ potentials]{Spectral properties of Schr\"odinger operators with locally $H^{-1}$ potentials}

\author[M.\ Luki{\'c}]{Milivoje Luki{\'c}}
\address{ Department of Mathematics, Rice University, Houston, TX 77005, USA}
\email{milivoje.lukic@rice.edu}
\thanks{M.L.\ was supported in part by NSF grants DMS--1700179 and DMS--2154563.}

\author[S.\ Sukhtaiev]{Selim Sukhtaiev}
\address{Department of Mathematics and Statistics,
	Auburn University, Auburn, AL 36849, USA}
\email{szs0266@auburn.edu}
\thanks{S.S. was supported in part by the Research Support Program grant provided by the Office of the Vice President for Research and Economic Development at Auburn University}

\author[X.\ Wang]{Xingya Wang}
\address{ Department of Mathematics, Rice University, Houston, TX 77005, USA}
\email{xw62@rice.edu}


\date{\today}

\begin{abstract}
We study half-line Schr\"odinger operators with locally $H^{-1}$ potentials. In the first part, we focus on a general spectral theoretic framework for such operators, including a Last--Simon-type description of the absolutely continuous spectrum and sufficient conditions for different spectral types. In the second part, we focus on potentials which are decaying in a local $H^{-1}$ sense; we establish a spectral transition between short-range and long-range potentials and an $\ell^2$ spectral transition for sparse singular potentials. The regularization procedure used to handle distributional potentials is also well suited for controlling rapid oscillations in the potential; thus, even within the class of smooth potentials, our results apply in situations which would not classically be considered decaying or even bounded.
\end{abstract}

\maketitle

{\scriptsize{\tableofcontents}}


\section{Introduction}

Schr\"odinger operators in one dimension $H_V = - \frac{d^2}{dx^2} +V$ are often considered in the setting of locally $L^2$ or locally $L^1$ potentials; however, there are several reasons to investigate more general potentials. One is the ubiquity of non-integrable singularities such as Coulomb- or $\delta$-type potentials in models from mathematical physics; another is the Lax pair representation of the KdV equation, where $H^{-1}(\bbR)$ and $H^{-1}(\bbT)$ is the optimal regularity for well-posedness \cite{MR2267286,MR3990604}. Non-integrable singularities are often studied by specialized methods such as those for the Kronig--Penney model, and inverse scattering arguments in the distributional setting are considered in ways that circumvent the underlying Schr\"odinger operators. One of the goals of this paper is to extend some robust techniques in spectral theory to the greater generality of locally $H^{-1}$ potentials, defined precisely below.

Schr\"odinger and Sturm--Liouville operators with distributional coefficients are often treated via the regularization method introduced in the pioneering work of Savchuk, Shkalikov \cite{MR1756602}. This approach has materialized into the main tool in the spectral theory of ordinary differential operators with measure and distributional coefficients. Indeed, it was employed, for example, by  Eckhardt, Teschl \cite{ET13} in the setting of measure coefficients; by Eckhardt, Gesztesy, Roger, Teschl \cite{MR3145132, MR3046408, EGNT13} for $L^1_{\loc}((a,b))$ four coefficient Sturm–Liouville operators; by Eckhardt, Kostenko, Malamud, Teschl \cite{MR3200376} for $\delta'$ potentials supported on Cantor sets; by Hryniv, Mykytyuk \cite{HrMyk01, MR2978191} for periodic singular potentials $H^{-1}_{\text{per}}(\bbR)=H^{-1}(\bbT)$; and by many other authors, see \cite{EGNT13} for an extensive reference list. Most of the papers in this direction address foundational questions such as self-adjointness, Weyl--Titchmarsh theory, spectral decomposition, as well as some inverse spectral problems. We emphasize that the study of spectral types such as in the current paper, and of the associated dynamics for operators with singular coefficients, have received much less attention and have been mostly restricted to periodic \cite[III.2.3]{MR2105735} and some ergodic \cite{MR4207172, MR4117490, MR4217910, MR797277}  Hamiltonians modeling point interactions.

In particular, Hryniv--Mykytyuk  \cite{HrMyk01, MR2978191} introduced a class of uniformly locally $H^{-1}$ potentials on $\bbR$ by the condition
\[
\sup_{n} \lVert V \phi_n \rVert_{H^{-1}(\bbR)} < \infty,
\]
with the help of compactly supported $H^{1}$ multipliers
\[
\phi_n(x) = \begin{cases}
 1 - 2 \lvert x-n \rvert^2  & \lvert x - n \rvert \le 1/2 \\
2 (\lvert x - n \rvert - 1 )^2 & 1/2 < \lvert x - n \rvert \le 1 \\
0 & 1 < \lvert x - n \rvert 
\end{cases}
\]
and showed that real distributions in this class are precisely those with a representation
\[
V = \sigma' + \tau,
\]
where $\sigma, \tau$ are real-valued functions on $\bbR$ such that
\begin{equation}\label{sigmatauboundedness}
\sup_x \int_x^{x+1} \sigma(t)^2 \,dt < \infty, \qquad \sup_x \int_x^{x+1} \lvert \tau(t) \rvert \,dt < \infty.
\end{equation}
Note that this class includes the potentials $V \in H^{-1}(\bbR)$ and $V \in H^{-1}(\bbT)$ (when viewed as periodic distributions on $\bbR$). In particular, the study of Schr\"odinger operators with locally $H^{-1}$ potentials helps to bridge spectral theory with scattering arguments. 
This decomposition is related to the Miura transformation and the Riccati representation \cite{Korotyaev,KappelerTopalov} for periodic $V$, in which every $V \in H^{-1}(\bbT)$ with zero average is represented uniquely in the form $V = \sigma' + \sigma^2 - \int_\bbT \sigma^2(t) \,dt$. In the non-periodic case, in the construction of \cite{HrMyk01}, $\tau$ takes the role of a local average, so the decomposition really requires two functions.

Several classes of singular potentials are modeled by a suitable choice of $\sigma, \tau$.  For example, a   Coulomb-type term $|x-x_0|^{-1}$, $x_0\in (0,\infty)$ is realized by setting $\sigma(x)=\log |x-x_0|$, $\tau(x)=0$, and the point interaction $\delta(x-x_0)$ is realized by the characteristic function $\sigma(x)=\chi_{[x_0,\infty)}$ and $\tau(x)=0$. 

\begin{remark}\label{remarklocalsize}
Of course, the decomposition $V = \sigma' + \tau$ is not unique; the procedure in \cite{HrMyk01} provides $\sigma$, $\tau$ such that 
\[
C^{-1} \sup_x \left( \lVert \sigma \chi_{[x,x+1)} \rVert_2 + \lVert \tau\chi_{[x,x+1)} \rVert_1 \right) \le \lVert V \rVert_{H^{-1}_{\unif}(\bbR)} \le C \sup_x \left( \lVert \sigma \chi_{[x,x+1)} \rVert_2 + \lVert \tau\chi_{[x,x+1)} \rVert_1 \right)
\]
with some universal constant $C$ (the second inequality is general; the first is a consequence of the choice of $\sigma,\tau$ starting from $V$). Accordingly, the quantity $\lVert \sigma \chi_{[x,x+1)} \rVert_2 + \lVert \tau\chi_{[x,x+1)} \rVert_1$ is interpreted as the local size of the potential.
\end{remark}

By Dirichlet decoupling and Weyl matrix arguments, many spectral properties of Schr\"odinger operators on $\bbR$ are reduced to spectral properties of half-line Schr\"odinger operators. For this reason, spectral properties are often naturally considered in the half-line setting. In this paper, we consider half-line Schr\"odinger operators with real-valued distributional potentials $V = \sigma' + \tau$. The formal rewriting
\[
-u'' + V u = - (u' - \sigma u)' - \sigma u' + \tau u = - (u' - \sigma u)' - \sigma (u' - \sigma u) + ( \tau - \sigma^2) u
\]
produces Schr\"odinger operators as follows:

\begin{hypothesis}\lb{pot} Denote $\bbR_+ = (0,\infty)$ and assume that $\sigma, \tau :\bbR_+\rightarrow\bbR$ obey \eqref{sigmatauboundedness}. Let
 \[
 u^{\1}:=u'-\sigma u
 \]
 denote the quasi-derivative of $u\in \AC(\bbR_+)$ and introduce 
\begin{align}
&\mathfrak D:=\{u\in \AC(\bbR_+): u^{\1}\in \AC(\bbR_+)\},\\
&\ell u:=-(u^{\1})' - \sigma u^{\1} + (\tau-\sigma^2)u,\qquad u\in\mathfrak D. \lb{difex}
\end{align} 
\end{hypothesis}

This leads to self-adjoint operators on the Hilbert space $L^2(\bbR_+)$ with a regular endpoint at $0$, limit point at $\infty$, given by
\[
 \dom(H^\alpha):=\{u\in L^2(\bbR_+): u\in\mathfrak D,\, \ell u\in L^2(\bbR_+), \, u(0)\cos(\alpha) +u^{\1}(0)\sin(\alpha)=0 \}
\]
where $\alpha$ parametrizes the boundary condition at $0$. We will discuss their self-adjointness and corresponding quadratic forms in Section~\ref{section2}. Note that this is consistent with standard ways of defining the operator if the potential is locally integrable (corresponding to $\sigma=0$) or with $\delta$-singularities (corresponding to jumps in $u'$), see \cite{HrMyk01}.

Using the quasi-derivative, the eigenfunction equation can be written as a first-order system for $\binom{u^{[1]}}u$. This is encoded by a family of transfer matrices $T(z,x)$ which is locally absolutely continuous in $x$ and solves the initial value problem
\[
\partial_x T(z,x) =\begin{pmatrix}
-\sigma(x) & \tau(x) - \sigma(x)^2 - z\\
1 & \sigma(x)
\end{pmatrix} T(z,x), \qquad T(z,0) = I.
\]
There is a corresponding Weyl function $m_\alpha$ and a canonical spectral measure $\mu^\alpha$. We will provide all definitions in Section 2; for the purpose of this introduction, it suffices to know that $\mu^\alpha$ is a maximal spectral measure for $H^\alpha$, and we are using it to make precise statements about the spectral type of $H^\alpha$. We will use the Lebesgue decomposition
\[
\mu^\alpha = \mu^\alpha_\ac + \mu^\alpha_\singc + \mu^\alpha_\pp.
\]

One of the goals of this paper is to establish sufficient conditions for different spectral types, including a criterion for a.c.\ spectrum which extends the results of  Last--Simon \cite{MR1666767} for locally integrable $V$. One is a description of an essential support for the a.c.\ spectrum in terms of Cesar\`o-boundedness of the transfer matrices: 

\begin{theorem}\lb{thm2.10} Assume Hypothesis \ref{pot}. Then, for arbitrary $\alpha\in[0,\pi)$, the set
\begin{align}
\Sigma_{\ac}:={\left\{E\in\bbR\Big| \liminf_{l\rightarrow\infty}\frac{1}{l}\int_{0}^{l}\|T(E;x)\|^2dx<\infty\right\}} \lb{sigmaac}
\end{align}
is an essential support for the a.c.\ spectrum of $H^\alpha$ in the sense that $\mu^\alpha_\ac$ is mutually absolutely continuous with the measure $\chi_{\Sigma_\ac}(E) \,dE$. In particular,
\[
\spec_{\ac}(H^{\alpha})=\overline{\Sigma_{\ac}}^{\ess}.
\]
\end{theorem}

Above we denoted the essential support of a Borel set $S$ by
\[
\overline{S}^{\ess}:=\{E\in\bbR: |S\cap(E-\varepsilon, E+\varepsilon)|>0, \forall \varepsilon>0\}.
\]
A closely related result gives a sufficient criterion for absence of a.c.\ spectrum:

\begin{theorem}\lb{cor2.11}
	Assume Hypothesis \ref{pot} and fix arbitrary $\alpha\in[0,\pi)$. Let $\cF\subset \bbR$ be a measurable set and suppose there exist sequences $\{x_j\}_{j=1}^{\infty}\subset \bbR_+, \{y_j\}_{j=1}^{\infty}\subset \bbR_+$ such that for Lebesgue almost every $E\in\cF$,
	\begin{equation}\lb{tnjass}
	\lim\limits_{j\rightarrow\infty} \|T(E;x_j, y_j)\|=\infty.
	\end{equation}
	Then, $\mu^{\alpha}_{ac}(\cF)=0$.
\end{theorem}

In the other direction one has:

\begin{theorem}\lb{Lpbiggerthan2criterion}
Assume Hypothesis \ref{pot} and fix $\alpha\in [0,\pi)$. Suppose that for some $p>2$,
\begin{equation}\lb{tpnew}
\liminf_{x\rightarrow\infty}\int_{E_1}^{E_2}\|T(E;x)\|^pdE<\infty. 
\end{equation}
Then, $H^{\alpha}$ has purely absolutely continuous spectrum on $(E_1, E_2)$. 
\end{theorem}

Theorems~\ref{thm2.10}, \ref{cor2.11}, \ref{Lpbiggerthan2criterion} generalize results of Last--Simon~\cite{MR1666767}. The proofs are given in Section~\ref{section2}, which also includes a Carmona-type formula, subordinacy, and a Simon--Stolz criterion for absence of eigenvalues.

An important ingredient are new pointwise eigenfunction estimates which are stated and derived in Section 2. These relate the pointwise behavior of a formal eigenfunction and its derivative to its local $L^2$ behavior. For $V\in L^2_{\loc}$ they follow from Sobolev embedding theorems, but for $V\notin L^2_{\loc}$, the local domain becomes $V$-dependent and different arguments are needed; estimates of this form were previously considered for locally $L^1$ potentials \cite{Stolz92,LukicDerivatives}. The pointwise estimates are given in Lemma~\ref{lem2.6}; here we point out one corollary of these estimates:
\begin{theorem}\label{theoremweightedL2eigenfunction}
 Assume Hypothesis \ref{pot} and let $w: (0,\infty) \to (0,\infty)$ obey 
 \begin{equation}\label{goodweight}
\sup_{\{x, y: |x-y| \le 1\}} \frac{w(y)}{w(x)} < \infty.
\end{equation}
For any $E\in \bbR$ there exists a positive constant $C$ such that for any $l>1$ and any solution $u\in \mathfrak D$, $\ell u=Eu$, one has
\begin{equation}\lb{derest1}
\int_1^l w(x) \lVert \vec u(x) \rVert^2 \dd x \le C \int_0^{l+1} w(x) \lvert u(x) \rvert^2 \dd x.
\end{equation}
In particular, if
\[
\int_0^\infty w(x) \lvert u(x)\rvert^2 \,dx < \infty,
\]
then
\[
\int_0^\infty w(x) \lvert u^{[1]}(x)\rvert^2 \,dx < \infty
\]
and
\begin{equation}\label{pointwisedecay}
\lim_{x\to\infty} \sqrt{w(x)} \lvert u(x) \rvert  = \lim_{x\to\infty} \sqrt{w(x)} \lvert u^{[1]}(x) \rvert = 0.
\end{equation}
\end{theorem}
In this paper, we will only use the case $w=1$; however, polynomial weights $w(x) = (x+1)^c$ and exponential weights $w(x) = e^{c x}$ for $c \in \bbR$ are also relevant for various criteria about the spectrum, spectral type, and dynamical properties which we expect to have a generalization to the current setting.

Remark~\ref{remarklocalsize} indicates that decay at $\infty$ should be quantified by the local $L^2$-norm on $\sigma$ and local $L^1$-norm on $\tau$. Thus, the following result generalizes the Blumenthal--Weyl criterion for preservation of essential spectrum under decaying perturbations:

\begin{lemma}\lb{lem2.4}
Assume Hypothesis \ref{pot} and suppose that 
\begin{align}\lb{unifzer}
\lim_{x\rightarrow\infty} \int_x^{x+1} \left( \sigma^2(t)+|\tau(t)| \right) \ dt =0.
\end{align}
Then, for arbitrary $\alpha\in[0,\pi)$, $\spec_{\ess}(H^{\alpha})=[0,\infty)$. 
\end{lemma}

We note in particular that Lemma~\ref{lem2.4} gives a more robust criterion even for locally $L^1$-potentials. Any locally uniformly $L^1$ potential $V$ can be decomposed as $\sigma = 0$, $\tau = V$, but choosing a different decomposition can give better results. For instance Lemma~\ref{lem2.4} implies:

\begin{corollary} \label{corollaryL1locessspec}
If $V \in L^1_{\loc}([0,\infty))$ is real-valued and the limit 
\[
\lim_{x\to\infty} \int_0^x V(t) \,dt
\]
is convergent, then the operator $-\frac{d^2}{dx^2}+V$ is limit point at $\infty$ and its arbitrary self-adjoint realization $H_V$ in $L^2(\bbR_+)$ satisfies $\spec_{\ess}(H_V) = [0,\infty)$. 
\end{corollary}

This corollary applies to oscillatory potentials such as
\begin{equation}\lb{bosc}
V(x):=(-1)^{\lfloor 2n(x-n)\rfloor}, x\in [n-1, n), n\in\bbN
\end{equation}
which was considered in \cite{EichingerLukic} by a more specialized argument, and to potentials
\begin{equation}\lb{unbosc}
V(x):=x^{\alpha}\sin(x^{\beta}),\  \alpha \ge 0, \beta>\alpha+1
\end{equation}
which aren't even locally uniformly integrable if $\alpha > 0$. Similar growing oscillatory potentials were considered in \cite{MR682723, DamanikKillip}. 

The description of the essential spectrum is the starting point in the theory of Schr\"odinger operators with decaying potentials, which are a classical subject and have been extensively studied over the past 30 years  \cite{MR397194, MR1101267, MR797277, MR2307748, MR1697600, MR585593, MR2552106,  MR1628290, MR931664, MR484145, MR1463044, vonNeumann1993, MR213915, MR682723}. Their spectral properties show a subtle competition between the rate of decay (with faster decay leading to absolutely continuous spectrum) and the disorder and oscillation in the potential (which promote more singular spectrum). Spectral transitions dependent on the rate of the decay have been studied by many authors, in particular: Pearson \cite{MR484145} in deterministic setting; Kiselev, Last, Simon \cite{MR1628290}, central to this paper; Delyon, Simon, Souillard \cite{MR797277} for discrete Schr\"odinger operators and Kronig--Penney models with decaying random potentials; and Kotani, Ushiroya \cite{MR931664} for continuous Schr\"odinger operators with decaying random potentials. This collection of papers gave rise to a number of subsequent investigations many of which are referenced in the review paper by Denisov, Kiselev \cite{MR2307748}. 

We first prove that short-range perturbations preserve pure a.c.\ spectrum.  In situations where different exponents are used to control local integrability and decay, the spaces of functions
\[
\ell^p(L^q) = \left \{ f : \bbR_+ \to \bbC \mid \sum_{n=0}^\infty \lVert f \chi_{[n,n+1)} \rVert_q^p < \infty \right\}
\]
are useful, cf. \cite{ChristKiselev,Rybkin04,Rybkin06}. The classical result about short-range perturbations is that $V\in L^1(\bbR_+)$ implies purely a.c.\ spectrum on $(0,\infty)$. The distributional analog of this criterion, informally speaking, is $\ell^1(H^{-1})$; following Remark~\ref{remarklocalsize}, we find the correct formulation.

\begin{theorem}\label{theorem.shortrange0}
Assume Hypothesis \ref{pot}. If $\sigma \in \ell^1(L^2)$ and $\tau \in \ell^1(L^1) = L^1(\bbR_+)$, then $H^\alpha$ has purely a.c.\ spectrum on $(0,\infty)$ for every $\alpha \in [0,\pi)$. 
\end{theorem}

In fact, we prove a more general result than Theorem~\ref{theorem.shortrange0}:

\begin{theorem}\label{theorem.shortrange}
	Assume Hypothesis \ref{pot} and
	\begin{align}\lb{firstac}
		\sigma\in L^1(\bbR_+), \quad (\sigma^2-\tau)\in L^1(\bbR_+).
	\end{align}
	Then for arbitrary $\alpha\in[0, \pi)$, the spectral measure on $(0,\infty)$ is of the form 
	\begin{equation}\lb{measconv}
	\chi_{(0,\infty)}(E)d\mu^\alpha(E) = w_\alpha(E) \,dE	
\end{equation}
		with $w_\alpha$ continuous on $(0,\infty)$ and strictly positive there. In particular, the spectrum of $H^{\alpha}$ is purely absolutely continuous on $(0,\infty)$.
\end{theorem}

To see that Theorem~\ref{theorem.shortrange} implies Theorem~\ref{theorem.shortrange0}, note that $\ell^1(L^2) \subset \ell^2(L^2) = L^2(\bbR_+)$ and $\ell^1(L^2) \subset \ell^1(L^1) = L^1(\bbR_+)$; thus, $\sigma \in \ell^1(L^2)$ and $\tau \in L^1(\bbR_+)$ implies \eqref{firstac}. These results apply, for instance, to potentials \eqref{unbosc} with $\beta > \alpha + 2$.

We note that neither condition in these theorems can be relaxed. For $\sigma = 0$ it is well-known that decay of $\tau$ weaker than $L^1$ can introduce singular spectrum in $(0,\infty)$; e.g., Wigner–von Neumann type potentials \cite{vonNeumann1993,MR247300,MR875178, MR1346989,Lukic2013} can exhibit eigenvalues embedded into ac spectrum with $\tau(x) = \cO(1/x)$ as $x \to \infty$. Similarly, we note that:

\begin{example}\lb{ex1.10}
There exists $\sigma \in \AC([0,\infty))$ with $\sigma(x) = \cO(1/x)$ as $x \to \infty$ such that for $\tau = 0$ and some $\alpha \in [0,\pi)$, the spectrum of $H^{\alpha}$ is not purely absolutely continuous on $(0,\infty)$.
\end{example}

Since such an example obeys $\sigma \in L^2(\bbR_+)$, it shows that the condition $\sigma \in \ell^1(L^2)$ cannot be relaxed in Theorem~\ref{theorem.shortrange0} and that the condition $\sigma^2 - \tau \in L^1(\bbR_+)$ cannot be relaxed in Theorem~\ref{theorem.shortrange}.

In the second part of the paper, we specialize to decaying sparse potentials and prove the following theorem.

	\begin{theorem}\lb{sparseAltnewtheorem} 
		Let $W_n \in H^{-1}(\bbR)$ be real distributions with $\supp W_n \subset [-\Delta,\Delta]$.  Assume that $W_n \to W$ in $H^{-1}(\bbR)$, with $W \neq 0$.  Let $d_n \to 0$,  let $\{x_n\}_{n=1}^{\infty}\subset \bbR_+$ be a monotonically increasing sequence such that $x_1 > \Delta$ and $\frac{x_n}{x_{n+1}}\rightarrow0$, and let
		\begin{equation}\lb{hmp}
V(x) =\sum_{n=1}^\infty d_n W_n(x- x_n).
		\end{equation}
		For any $\alpha \in [0,\pi)$, $\spec_{\ess}(H^{\alpha})=[0, \infty)$ and moreover:

(a)  If $\sum_{n=1}^{\infty} \lvert d_n\rvert^2 <\infty$ then the spectrum of $H^{\alpha}$ is purely absolutely continuous on  $(0,\infty)$, in the sense that $\chi_{(0,\infty)}\,d\mu^\alpha$ is mutually absolutely continuous with Lebesgue measure on $(0,\infty)$. In particular, 	
			$\spec_{\singc}(H^{\alpha})=\emptyset$, $\spec_{\pp}(H^{\alpha})\subset(-\infty,0]$, $\spec_{\ac}(H^{\alpha})=[0,\infty)$. 	

(b) If $\sum_{n=1}^{\infty} \lvert d_n\rvert^2 = \infty$, then the spectrum of $H^{\alpha}$ is purely singular continuous on $[E_1,E_2]$. 
\end{theorem}

The special choice $W_n \equiv W\in L^{\infty}((-\Delta, \Delta))$ yields Pearson-type classical potentials; that case of Theorem \ref{sparseAltnewtheorem} was proved by Kiselev, Last, Simon \cite{MR1628290}. Our extension allows more singular potentials; for instance, as an illustration of Theorem \ref{sparseAltnewtheorem}, we claim a Kiselev--Last--Simon-type spectral transition  for the  Kronig--Penney model. Concretely, let $H$ be the Schr\"odinger operator acting on $L^2(\bbR_+)$ given by
\begin{equation}\lb{Hop}
H=-\frac{d^2}{dx^2}+\sum_{n=1}^{\infty} d_n \delta (x-x_n),
\end{equation} 
where $\{x_n\}_{n=1}^{\infty}\subset (0,\infty)$ is a sparse sequence satisfying $ x_n/x_{n+1}\rightarrow 0$ as $n\rightarrow\infty$, subject to any self-adjoint condition at $0$. Then for any decaying sequence $(d_n)_{n=1}^\infty$,  $\spec_{\ess}(H)=[0,\infty)$; moreover, the spectrum is purely a.c.\ on $(0,\infty)$ if $(d_n)_{n=1}^\infty$ is square-summable and purely s.c.\ on $(0,\infty)$ otherwise.

Another new feature of our result is that the profile $W_n$ may vary with $n$. Note that this allows examples such as the locally integrable potential
\[
V = \sum_{n=1}^\infty d_n n \chi_{[x_n, x_n + 1/n]} 
\]
where $\chi$ denote characteristic functions. Since $n\chi_{[0,1/n]} \to \delta_0$ in $H^{-1}(\bbR)$,  by Theorem~\ref{sparseAltnewtheorem}, the spectrum is purely a.c.\ on $(0,\infty)$ if the decaying sequence $(d_n)_{n=1}^\infty$ is square-summable and purely s.c.\ on $(0,\infty)$ otherwise.

Although stated in terms of $H^{-1}(\bbR)$,  the starting point in our analysis is a decomposition $W_n=S_n'+T_n$ and the proof must treat these contributions to $\sigma$ and $\tau$ separately. As in the classical case \cite{MR1628290} our proof is based on the analysis of Pr\"ufer variables. However, in the present case this analysis is more intricate due to the appearance of new terms in the differential equations obeyed by Pr\"ufer variables. Namely, in the setting of $H^{-1}$ potential $V=\sigma'+\tau$, as shown in Proposition \ref{prop2.4new}, one has
	\begin{align}\lb{pruf1}
	&\theta' = k - \frac{\tau -\sigma^2}{k} \sin^2(\theta) + \sigma\sin(2\theta),\qquad (\log R)'  = \frac{\tau-\sigma^2}{2k}\sin(2\theta) - \sigma\cos (2\theta),
\end{align}
whereas in the classical case $V\in L^1_{\loc}(\bbR_+)$, as discussed in \cite{MR1628290}, 
\begin{equation}\lb{pruf2}
\theta' = k - \frac{V}{k} \sin^2(\theta),\qquad   (\log R)'  = \frac{V}{2k}\sin(2\theta). 
\end{equation}
An important ingredient in the proof of Theorem \ref{sparseAltnewtheorem} is given by the fact that $\|T(k^2, x)\|$ is comparable to $R(x)$, see Proposition \ref{prop2.4new}. Hence, in order to establish growth or boundedness of eigensolutions and, respectively, the absence or existence of purely absolutely continuous spectrum on $[E_1, E_2]$, it suffices to study the asymptotics for $R(x)$. In Sections \ref{secac} and \ref{secsc}, we describe the asymptotic behavior of $R(x)$  depending on whether or not $\{d_n\}_{n=1}^{\infty}\in\ell^2(\bbN)$. 

\textbf{Acknowledgments:} The authors thank David Damanik and Fritz Gesztesy for useful discussions. Selim Sukhtaiev and Xingya Wang gratefully acknowledge support from the Simons Center for Geometry and Physics, Stony Brook University, where a part of this research was completed during the workshop "Ergodic Operators and Quantum Graphs".

\section{Spectral analysis of Schr\"odinger operators with distributional potentials} \label{section2}

In this section, we consider Schr\"odinger operators in the setting of Hypothesis~\ref{pot}.

\subsection{Self-adjointness and form bounds}
Associated with the differential expression $\ell$ are three linear, densely defined, unbounded operators $H_0, H_{\min}, H_{\max}$ acting on $L^2(\bbR_+)$ defined as follows:
\begin{align}
&H_{\max} u=\ell u,\qquad u\in \dom(H_{\max}):=\{u\in L^2(\bbR_+): u\in\mathfrak D,\, \ell u\in L^2(\bbR_+) \},\\
&H_{0} u= \ell u,\qquad u\in \dom(H_{0}):=\{u\in \dom(H_{\max}): u\text{\ has compact support}\},
\end{align}
and $H_{\min}:=\overline{H_0}$, the closure of $H_0$ in $L^2(\bbR_+)$. Then, upon setting $q=\tau, p=1, r=1, s=\sigma$ in \cite[Section 3]{EGNT13}, we infer
\begin{equation}\lb{hs}
	H_{\min}=\overline{H_0}=H_0^{**}=H_{\max}^*;\qquad H_{\min}\subset H_{\min}^*=H_{\max}\,.  
\end{equation} 
In the following Theorem, we discuss self-adjoint extensions of $H_{\min}$, prove that $\ell$ is limit point at $\infty$ and limit circle at $0$, and obtain auxiliary resolvent estimates. 

\begin{theorem}\lb{thmoper}
Assume Hypothesis \ref{pot}. Then, $\ell$ is limit point at infinity and limit circle at $0$. That is, for all $z\in\bbC$, every solution of $\ell u-zu=0$ lies in $L^2$ near zero, and there is one nontrivial solution that lies in $L^2$ near infinity and one solution that does not, up to scalar multiples. Moreover, for every $u\in\dom(H_{\max})$, the limits 
\begin{equation}\lb{uupr}
u(0)=\lim\limits_{x\rightarrow 0} u(x),\qquad u^{\1}(0)=\lim\limits_{x\rightarrow 0} u^{\1}(x)
\end{equation}
exist and are finite; they give rise to an explicit description of the minimal operator:
\begin{align}
	&H_{\min} u=-(u^{\1})' - \sigma u^{\1} + (\tau-\sigma^2)u,\qquad u\in \dom(H_{\min}),\\
	&\dom(H_{\min})=\{u\in \dom(H_{\max}): u(0)=u^{\1}(0)=0\},
\end{align}
where all self-adjoint extensions of $H_{\min}$ are parametrized by $\alpha\in[0,\pi)$ as follows:
 \begin{align}
 \begin{split}\lb{Halpha}
 &H^{\alpha} u=-(u^{\1})' - \sigma u^{\1} + (\tau-\sigma^2)u,\qquad u\in \dom(H^{\alpha}),\\
 &\dom(H^{\alpha})=\{u\in \dom(H_{\max}): u(0)\cos(\alpha) +u^{\1}(0)\sin(\alpha)=0\}.
 \end{split}
 \end{align}
The quadratic form $\mathfrak h^{\alpha}$ of $H^{\alpha}$ is given by
\begin{equation}
\mathfrak h^{\alpha}[u,v]=
\begin{cases}
\begin{matrix}
\langle u', v'\rangle_{L^2(\bbR_+)}-\langle \sigma u', v\rangle_{L^2(\bbR_+)}-\langle  u, \sigma v'\rangle_{L^2(\bbR_+)}+\\
\hspace{3cm}+\langle \tau u, v\rangle_{L^2(\bbR_+)}-\cot(\alpha) \overline{u(0)}v(0),
\end{matrix}&\alpha\in(0,\pi),\\
\quad\langle u', v'\rangle_{L^2(\bbR_+)}-\langle \sigma u', v\rangle_{L^2(\bbR_+)}-\langle  u, \sigma v'\rangle_{L^2(\bbR_+)}+\langle \tau u, v\rangle_{L^2(\bbR_+)},&\alpha=0.
\end{cases}
\end{equation}
for $u,v\in\dom(\mathfrak h^{\alpha})$, where
\begin{equation}
\dom(\mathfrak h^{\alpha}):=\begin{cases}
	H^1(\bbR_+),\qquad \alpha\in(0,\pi),\\
	H_0^1(\bbR_+):=\{f\in H^1(\bbR_+): f(0)=0\},\qquad \alpha=0.
\end{cases}
\end{equation}
Furthermore, the operator $H^{\alpha}$ is bounded from below and there exist $C=C(\sigma, \tau)>1, \lambda=\lambda(\sigma, \tau)>0$ such that for $E<\min\{{\inf\spec(H^{\alpha}), 0}\}$,
\begin{align}\lb{resineq}
(H^{\alpha}-E)^{-1}\leq C(-\Delta_X-E+\lambda)^{-1},
\end{align}
where $-\Delta_X$ is the Dirichlet Laplacian on $\bbR_+$ if $\alpha=0$ and Neumann Laplacian if $\alpha\in(0,\pi)$. 
\end{theorem}
\begin{proof}
By Hypothesis \ref{pot}, differential expression $\ell$ is regular at $0$; thus, by \cite[Lemma  3.1]{EGNT13}, all solutions of $\ell u-zu=0$ can be extended by continuity to $0$ so that $u, u^{[1]}$ are absolutely continuous in a neighborhood of $0$. Hence, all such solutions are square integrable near $0$ and $\ell$ is limit circle near $0$.

In this setting, the Wronskian is defined for $u, v \in \cD$ by
\[
W(u,v)(x) = u(x) v^{[1]}(x) - u^{[1]}(x) v(x)
\]
and in order to show that $\ell$ is limit point at infinity, it suffices to check that 
\begin{equation}\lb{wrzer}
\lim\limits_{x\rightarrow\infty} W(f, g)(x)=0,\qquad f,g\in\dom(H_{\max}).
\end{equation}
cf. \cite[Lemma 4.4]{EGNT13}. To that end, we will first prove  that every $f\in \dom(H_{\max})$, $f^{[1]}$  lies in $L^2(\bbR_+)$. Since $f^{[1]} \in \AC(\bbR_+)$, it suffices to analyze it near infinity. Let $\wti \sigma$, $\wti \tau$ be extensions of $\sigma, \tau$ by zero to the whole line $\bbR$ and consider the operator $\wti H$ acting on $L^2(\bbR)$ given by 
\begin{align}
	&\wti H u=\wti \ell u, \wti\ell u:=-(u'-\wti \sigma u)' - \wti\sigma(u'-\wti \sigma u) + (\wti \tau-\wti \sigma^2)u,\qquad u\in \dom(\wti H),\\
	&\dom(\wti H)=\{u\in L^2(\bbR): u\in\AC(\bbR),\,(u'-\wti \sigma u)\in \AC(\bbR),\,\ell u\in L^2(\bbR) \}.
\end{align} 
This operator is self-adjoint, bounded from below and its form domain is given by $H^1(\bbR)$, see \cite{HrMyk01, MR2978191}; in particular, $\dom(\wti H)\subset H^1(\bbR)$. Fix $u\in\dom(H_{\max})$ and let $\wti u$ be the extension of $u$ by zero to the whole line $\bbR$. Pick any $\phi\in C^{\infty}(\bbR)$ with $\supp(\phi)\subset(1/2,\infty)$ and $\phi(x)=1$ for $x\geq 1$. Then, $\wti u\phi\in \dom(\wti H)$ and, as $\dom(\wti H)\subset H^1(\bbR)$, one also has $\wti u\phi \in H^1(\bbR)$. Since $(\wti u\phi)'(x)=u'(x), x\geq 1$, we infer that $u'$ lies in $L^2$ near infinity.  Next, we show that $\sigma u$ lies in $L^2$  and $u$ lies in $H^1$ near infinity. Recall from \cite[Lemma 3.1]{HrMyk01} that for arbitrary interval $I\subset \bbR_+$ of length $1$, $\varepsilon\in(0,1)$, and $\psi\in H^1(I)$,
\begin{align}
&\|\psi\|_{L^{\infty}(I)}^2 \leq \varepsilon \|\psi'\|_{L^{2}(I)}^2+8\varepsilon^{-1}\|\psi\|_{L^{2}(I)}^2,\lb{sob1}\\
&\|\psi\psi'\|_{L^{2}(I)}\leq \varepsilon \|\psi'\|_{L^{2}(I)}^2+4\varepsilon^{-3}\|\psi\|_{L^{2}(I)}^2. \lb{sob2}
\end{align}
In particular, inspired by the proof of \cite[Theorem 3.4]{HrMyk01}, we get
\begin{align}\lb{fsig}
\int_0^{\infty}|\sigma u|^2=\sum_{n=0}^{\infty} \int_n^{n+1}|\sigma u|^2\leq \|\sigma\|_{2, \unif}^2\sum_{n=0}^{\infty} \|u\|_{L^{\infty}(n, n+1)}\underset{\eqref{sob1}}{\leq} C\|u\|_{H^1(\bbR)}^2<\infty.
\end{align}
That is, $\sigma u\in L^2(\bbR_+)$. Consequently, $u^{[1]}=u'-\sigma u \in L^2(\bbR_+)$, and by Cauchy--Schwarz, the Wronskian $W(u,v)$ lies in $L^1(\bbR_+)$. Moreover, since $W(u,v)(x)$ has a limit at infinity, see \cite[Lemma 3.2]{EGNT13}, it must converge to zero as asserted in \eqref{wrzer}. In conclusion, $\ell$ is limit point at infinity.

The fact that all self-adjoint extensions of $H_{\min}$ are determined by the boundary conditions \eqref{Halpha} follows from \cite[Theorem 6.2]{EGNT13} (where one should pick $BC^1_0(u):=u(0)$, $BC^2_0(u):=u^{\1}(0)$).

Let us now switch to quadratic form $\mathfrak h^{\alpha}$. Our first objective is to show that it is relatively bounded with respect to the quadratic form of the Dirichlet or Neumann free Laplacian on $\bbR_+$, depending on the value of $\alpha$. Note that for arbitrary $\varepsilon>0$, employing \eqref{sob1}, \eqref{sob2} as in the proof of \cite[Lemma 3.2]{HrMyk01},
\begin{align}
\begin{split}\lb{formbounds1}
&|\langle \sigma u', u\rangle_{L^2(\bbR_+)}|\leq \sum_{n=0}^{\infty} \int_{n}^{n+1}|\sigma \overline{u'}u|dx\underset{\eqref{sob2}}{\leq} \|\sigma\|_{2,\unif}(\varepsilon \|u'\|^2_{L^2(\bbR_+)}+4\varepsilon^{-3} \|u\|^2_{L^2(\bbR_+)}),\\
&|\langle \tau u, u\rangle_{L^2(\bbR_+)}|\leq \sum_{n=0}^{\infty} \int_{n}^{n+1}|\tau|dx\ \|u\|_{L^{\infty}(n, n+1)}\underset{ \eqref{sob1}}{\leq }\|\tau\|_{1,\unif}(\varepsilon \|u'\|^2_{L^2(\bbR_+)}+8\varepsilon^{-1} \|u\|^2_{L^2(\bbR_+)}),
\end{split}
\end{align}
for arbitrary $u\in H^1(\bbR_+)$; moreover, by \eqref{sob1},
\begin{equation}\lb{formbounds2}
|u(0)|^2\leq \varepsilon \|u'\|^2_{L^2(\bbR_+)}+4\varepsilon^{-3} \|u\|^2_{L^2(\bbR_+)}.
\end{equation}
Let $\mathfrak h^{X}$, $X\in\{D, N\}$ denote the quadratic form corresponding to Dirichlet or Neumann free Laplacian on $\bbR_+$; i.e., $\mathfrak h^{X}(u,u)=\|u'\|^2_{L^2(\bbR_+)}$, $u\in \dom(\mathfrak h^X)$, where $\dom(\mathfrak h^D):= H_0^1(\bbR_+)$ and $\dom(\mathfrak h^N):= H^1(\bbR_+)$.  We will proceed with assuming $\alpha\in(0, \pi)$, the second case $\alpha=0$ can be handled similarly.  For any $a\in(0,1)$, the inequalities \eqref{formbounds1}, \eqref{formbounds2} yield $b\in\bbR$ such that
\begin{align}
|\mathfrak \langle \sigma u', u\rangle_{L^2(\bbR_+)}&+\langle  u, \sigma u'\rangle_{L^2(\bbR_+)}-\langle \tau u, u\rangle_{L^2(\bbR_+)}+\cot(\alpha) \overline{u(0)}v(0)|\\
&\leq a\|u'\|^2_{L^2(\bbR_+)}+b\|u\|^2_{L^2(\bbR_+)},\qquad u\in H^1(\bbR_+).
\end{align}
That is, the lower order terms and the boundary term in the definition of $\mathfrak h^{\alpha}$, considered as quadratic form on $H^1(\bbR_+)$, are relatively bounded with respect to Neumann form $\mathfrak h^N$, with relative bound less than one, see \cite[Section VI.3.3]{K80} or \cite[Chapter X]{MR0493420}. Thus, by \cite[Theorem X.17]{MR0493420}, $\mathfrak h^{\alpha}$ is closed bounded from below quadratic form and there is a unique self-adjoint operator $T^{\alpha}$ acting in $L^2(\bbR_+)$ which satisfies
\begin{equation}
\langle T^{\alpha}u,v\rangle_{L^2(\bbR_+)}=\mathfrak h^{\alpha}(u,v),\qquad u\in H^2(\bbR_+),\,\, v\in \dom(T^{\alpha}). 
\end{equation}
We claim that $H^{\alpha}\subset T^{\alpha}$. Assume this claim, we note that both operators are self-adjoint and therefore must coincide. This implies that $\mathfrak h^{\alpha}$ is the quadratic form of the operator $H^{\alpha}$ which is consequently bounded from below. Returning to $ H^{\alpha}\subset T^{\alpha}$: let $u\in\dom(H^{\alpha})$ and $v\in H^1(\bbR_+)$; then,
\begin{align}
\mathfrak h^{\alpha}(u,v)&=\langle u^{\1}, v'\rangle_{L^2(\bbR_+)}-\langle \sigma u^{\1}, v\rangle_{L^2(\bbR_+)}+\langle (\tau -\sigma^2) u, v\rangle_{L^2(\bbR_+)}-\cot(\alpha) \overline{u(0)}v(0)\\
&=\langle -(u^{\1})'-\sigma u^{\1}+ (\tau -\sigma^2) u, v\rangle_{L^2(\bbR_+)}-\overline{u^{\1}(0)}v(0)-\cot(\alpha) \overline{u(0)}v(0)\\
&=\langle H^{\alpha}u, v\rangle_{L^2(\bbR_+)},
\end{align}
where in the second step, we used the boundary condition $u^{\1}(0)=-\cot(\alpha)u(0)$.

In order to prove \eqref{resineq} (again we focus on the case $\alpha\in(0,\pi)$), we invoke \eqref{formbounds1}, \eqref{formbounds2} to obtain some $C=C(\sigma, \tau)>1$ such that
\begin{equation}
\mathfrak h^{\alpha}(u,u)\leq C(\| u'\|^2_{L^2(\bbR_+)}+\lambda\|u\|^2_{L^2(\bbR_+)}),\qquad u\in H^{1}(\bbR_+). 
\end{equation}
Noting that the left-hand side above is the quadratic form of $H^{\alpha}$ and the right-hand side is the quadratic form of $C(-\Delta_N+\lambda)$, the assertion \eqref{resineq} follows from \cite[Theorem VI 2.21]{K80}, where it is shown that the ordering of quadratic forms implies the ordering of resolvents. 
\end{proof}

\begin{remark}\label{remarkgaugechange}
(i)  
The representation $V=\sigma'+\tau$ is not unique; given two pairs $(\sigma_i, \tau_i)\in L^2_{\loc}(\bbR_+)\times L^1_{\loc}(\bbR_+)$, $i=1,2$ with $\sigma_1'+\tau_1=V=\sigma_2'+\tau_2$ one has
\begin{equation}
\theta:=\sigma_1-\sigma_2, \qquad \theta'=\tau_2-\tau_1,
\end{equation}
so that $\theta\in W^{1,1}_{\loc}(\bbR_+)$. 

(ii) Fix $\theta\in W^{1,1}_{\loc}(\bbR_+)$ and $(\sigma, \tau)\in L^2_{\loc}(\bbR_+)\times L^1_{\loc}(\bbR_+)$. We say that the pair $(\sigma+\theta, \tau-\theta')$ is a gauge change of $(\sigma, \tau)$. The domain $\dom(H_{\max})$ is gauge change invariant since for $u\in \AC(\bbR_+)$ one has
\begin{equation}
(u'-\sigma u)\in \AC(\bbR_+) \iff (u'-(\sigma+\theta) u)\in \AC(\bbR_+)
\end{equation}
and a direct calculation shows that the action of the maximal operator $H_{\max}$ is also gauge change invariant. The gauge change affects the definition of the quasi-derivative $u^{[1]}_j = u' - \sigma_j u$ so that $u^{[1]}_1 = u^{[1]}_2 - \theta u$. Therefore the self-adjoint boundary conditions $u(0) \cos \alpha_j + u^{[1]}_j(0) \sin \alpha_j = 0$ are relabelled by the formula
\[
\cot \alpha_2  = \cot \alpha_1 - \theta(0).
\]
\end{remark}

\begin{remark}\lb{thm2.3}In the setting of Theorem \ref{thmoper},
\begin{equation}
\dim \rank\left((H^{\alpha}-\bfi)^{-1}-(H^{\beta}-\bfi)^{-1}\right)\leq 2.
\end{equation}
This is due to the fact that the deficiency indices of $H_{\min}$ are $(2,2)$ and the abstract Krein's resolvent formula \cite[Theorem A.1]{MR3202926}.
\end{remark}

We can now prove our version of the Blumenthal--Weyl criterion:

\begin{proof}[Proof of Lemma~\ref{lem2.4}]
By Remark \ref{thm2.3} and \cite[Theorem 2.4]{MR929030}, it suffices to prove the statement for $\alpha=0$.  Let $H_D$ and $\mathfrak h_D$ denote respectively the Dirichlet Laplacian and its quadratic form on $\bbR_+$; i.e., using the notation of Theorem \ref{thmoper} with $\alpha=0$, $\sigma=\tau=0$, write $H_D:=H^{0}$, $\mathfrak h_D=\mathfrak h^{0}$. Our goal is to show that for $\sigma, \tau$ as in \eqref{unifzer}, the quadratic form $\mathfrak h^{0}$ is a relative compact perturbation of $\mathfrak h_D$ (see e.g., \cite[Definition 2.12]{MR606197}, \cite[Section IV.4]{MR929030}). This assertion together with \cite[Theorem 2.13]{MR606197} yields $\spec_{\ess}(H^{0})=\spec_{\ess}(H_D)$ and, when combined with $\spec_{\ess}(H_D)=[0,\infty)$, proves the statement. 

Consider the quadratic form:
\begin{align}
&\mathfrak s[u,v]:=-\langle \sigma u', v\rangle_{L^2(\bbR_+)}-\langle  u, \sigma v'\rangle_{L^2(\bbR_+)}+\langle \tau u, v\rangle_{L^2(\bbR_+)},  u,v\in\dom(\mathfrak s)=H^1_0(\bbR_+). 
\end{align}
In order to show that $\mathfrak h^{0}$ is a relative compact perturbation of $\mathfrak h_D$, it suffices to verify
\begin{align}
&(i)\,\,\, |\mathfrak s[u,u]|\leq C (\mathfrak h^0[u,u]+\|u\|^2_{L^2(\bbR_+)})\text{ for any } u\in H^1_0(\bbR_+), \lb{condition1}\\
\begin{split}
&(ii) \text{ if }\sup_j\|u_j\|_{H^1(\bbR_+)}\leq 1\text{, then there exists a subsequence } \{u_{j_m}\}_{m=1}^{\infty}\text{ such that for $\varepsilon>0$}\\
&\text{ there exists $K>1$ such that }  |\mathfrak s [u_{j_m}-u_{j_n}, u_{j_m}-u_{j_n}]<\varepsilon\text{ for all } m,n>K\lb{condition2},
\end{split}
\end{align}
cf. \cite[Theorem 2.14]{MR606197}. The first inequality \eqref{condition1} follows from \eqref{formbounds1}, so it suffices to prove \eqref{condition2}. First, let $\chi_{[a,b]}$ denote the characteristic function of $[a,b]$ and note that
\begin{align}
\begin{split}\lb{sone}
|\mathfrak s[u,u]|&\leq 2|\langle \chi_{[0,t]}\sigma u', u\rangle_{L^2(\bbR_+)}|+|\langle  \chi_{[0,t]} \tau u, u\rangle_{L^2(\bbR_+)}|\\
&\quad + 2|\langle\chi_{[t,\infty)}\sigma u', u\rangle_{L^2(\bbR_+)}|+|\langle  \chi_{[t,\infty)}\tau u, u\rangle_{L^2(\bbR_+)}|,\qquad u\in \dom(\mathfrak s).
\end{split}
\end{align}
Fix arbitrary $\varepsilon>0$,  then for a sequence$\{u_j\}_{j=1}^{\infty}\subset\dom(\mathfrak s)$ with $\sup_j\|u_j\|_{H^1(\bbR_+)}\leq 1$ and sufficiently large, $j$-independent, $t=t(\varepsilon, \sigma, \tau)>0$,
\begin{align}
\begin{split}\lb{stwo}
&2|\langle \chi_{[t,\infty)}\sigma (u_j-u_k)', (u_j-u_k)\rangle_{L^2(\bbR_+)}|+|\langle  \chi_{[t,\infty)} \tau (u_j-u_k), (u_j-u_k)\rangle_{L^2(\bbR_+)}|\\
&\quad \underset{ \eqref{formbounds1}}{\leq} C (\|\chi_{[t,\infty)}\sigma\|_{2,\unif}+\|\chi_{[t,\infty)}\tau\|_{1,\unif} )\|u_j-u_k\|_{H^1(\bbR_+)}\leq \varepsilon/2,\quad \text{for all }j\in\bbN,
\end{split}
\end{align}
where in the last inequality, we used \eqref{unifzer} and $\sup_j\|u_j\|_{H^1(\bbR_+)}\leq 1$. Next, for $t$ defined above, note that  $\sup_j\|\chi_{[0,t]} u_j\|_{H^1(\bbR_+)}\leq 1$ and, due to compactness of the embedding $H^1((0,t))\hookrightarrow L^2((0,t))$, there exists a subsequence $\{u_{j_k}\}_{k=1}^{\infty}$ which is Cauchy in $L^2(\bbR_+)$. For such a subsequence and arbitrary $\varepsilon>0$, there exists $K>1$ such that 
\begin{align}
\begin{split}\lb{sthree}
&2|\langle \chi_{[0,t]}\sigma (u_{j_m}-u_{j_n})', \chi_{[0,t]}(u_{j_m}-u_{j_n})\rangle_{L^2(\bbR_+)}|\\
&\quad +|\langle  \chi_{[0,t]} \tau(u_{j_m}-u_{j_n}) ,(u_{j_m}-u_{j_n})\rangle_{L^2(\bbR_+)}|<\varepsilon/2,\qquad \text{for } m,n>K,
\end{split}
\end{align}
where we used the Cauchy--Schwartz inequality and $\sup_j\|\chi_{[0,t]} u_j\|_{H^1(\bbR_+)}\leq 1$. It follows from \eqref{sone} with $u:=u_{j_m}-u_{j_n}$, \eqref{stwo} and \eqref{sthree} that
\begin{equation}
s [u_{j_m}-u_{j_n}, u_{j_m}-u_{j_n}]|<\varepsilon,\qquad \text{for } m,n>K,
\end{equation}
which yields \eqref{condition2} as required. 
\end{proof}

\begin{proof}[Proof of Corollary \ref{corollaryL1locessspec}]
Let $\sigma(x):=\int_{0}^{x}V(t)dt-\int_{0}^{\infty}V(t)dt$, $\tau=0$. Then $\sigma(x)\rightarrow 0$, $x\rightarrow\infty$, hence, \eqref{sigmatauboundedness} holds, hence, by Theorem \ref{thmoper} $H_V$ is limit point at infinity. In addition one has \eqref{unifzer}, thus, by Lemma \ref{lem2.4}, $\spec_{\ess}(H^{\alpha})=[0,\infty)$ which combined with $\sigma'+\tau=V$ yield $\spec_{\ess}(H_V)=[0,\infty)$ as asserted. 
\end{proof}

At this point let us prove the assertion made in {Example \ref{ex1.10}}. 

\begin{proof}[Proof of Example~\ref{ex1.10}]
The Wigner--von Neumann potential $V$, explicitly defined in \cite[Section 3, Part B]{MR247300}, admits a real-valued nontrivial eigenfuction  $u\in L^2(\bbR_+)$ corresponding to eigenvalue $1$. In particular, for the choice of boundary condition at $0$ corresponding to $u$,  the Schr\"odinger operator $-\frac{d^2}{dx^2}+V$ does not have purely absolutely continuous spectrum on $(0,\infty)$. We set  $\sigma(x):=-\int_{x}^{\infty}V(t)dt$ and $\tau=0$. Then $V = \sigma'+ \tau$ so this is a gauge change of the Wigner--von Neumann potential; in particular, spectral type is unchanged. To prove $\sigma(x)\underset{x\rightarrow\infty}{=}\cO(1/x)$, we recall the asymptotic formula 
\begin{equation}\lb{Wig}
V(t)=-\frac{8\sin(2t)}{t}+\cO(t^{-2}),\ \  t\rightarrow\infty.
\end{equation}
Hence, for some $C, c>0$ and sufficiently large $x$ we have
\begin{align}
\left|\int_{x}^{\infty}V(t)dt\right|\leq \left| \int_{x}^{\infty}\frac{8\sin(2t)}{t}dt \right|+\left|\frac{c}{x}\right|\leq  \left|\frac{4\cos(2x)}{x}\right|+ \left| \int_{x}^{\infty}\frac{4\cos(2t)}{t^2}dt \right|+ \frac{c}{x}\leq \frac{C}{x}.
\end{align}
\end{proof}

\begin{remark} (i) The invariance of the essential spectrum under small at infinity perturbations of the coefficients has been investigated by many authors in various settings, see e.g., \cite{MR1918787, MR1361167, MR0126729} and especially \cite{MR2254485}, which contains many relevant references. The central fact in the classical treatment of this problem via Weyl-type sequences, see \cite[Section 10]{MR1361167}, is that $H^{\alpha}$ has a locally compact resolvent; i.e., $\chi_{(0,n)}(H^{\alpha}-\bfi)^{-1}, n\geq 0$ is compact in $L^2(\bbR_+)$. This still holds in our case, as readily seen from the explicit form of Green's function. However, there is a major obstacle in using the classical approach since $\dom(H^{\alpha})$, as a subset of $L^2(\bbR_+)$, depends on $\sigma, \tau$. Notably, one does not even have the inclusion $C_0^{\infty}(\bbR_+)\subset \dom(H^{\alpha})$ in general; e.g., such an inclusion does not hold when $V$ is not locally $L^2$. The key feature of our proof of Lemma \ref{lem2.4} is that the form domain $\mathfrak h^{\alpha}$ does not depend on $\sigma, \tau$. Interestingly, the latter does depend on $\alpha$, though the invariance of essential spectrum under perturbation of the boundary condition is handled by Krein's formula for the difference of resolvents of two self-adjoint extensions of the minimal operator $H_{\min}$, as discussed in Remark \ref{thm2.3}. 
	
(ii) Relevant to this discussion is  \cite[Theorem 3.2]{MR3134550} (see also \cite{MR2457601}, \cite{MR2198326}), where the full-line version of \eqref{unifzer} is shown to be equivalent to compactness of the multiplier given by an $H^{-1}(\bbR)$ potential. 
\end{remark}

\subsection{Weyl-Titchmarsh theory}  Let us fix $z\in\mathbb{C}$, $g\in L^1_{\loc}(\bbR_+)$ and consider the differential equation
\begin{equation}\lb{specpr}
-(u^{\1})' - \sigma u^{\1} + (\tau-\sigma^2)u - zu=g,\qquad u\in\mathfrak D.
\end{equation}
Rewriting it as a first order system, we get 
\begin{equation}\lb{amat}
\frac{d}{dx}\begin{bmatrix}
	u^{\1}(x)\\
	u(x)
\end{bmatrix} =A(z,x)\begin{bmatrix}
	u^{\1}(x) \\
	u(x)
\end{bmatrix} -
\begin{bmatrix}
	g(x) \\
	0
\end{bmatrix},\qquad
A(z,x):=\begin{bmatrix}
-\sigma & (\tau - \sigma^2) - z\\
1 & \sigma
\end{bmatrix}.  
\end{equation}
Assuming Hypothesis \ref{pot}, since the matrix coefficients in $A(z,x)$ lie in $L^1_{\loc}(\bbR; \bbC^{2\times 2})$, the corresponding initial value problem has a unique locally absolutely continuous solution.  In particular, for $g=0$, $\alpha\in[0,\pi)$, we consider the initial value problem $\ell u-zu=0$ and denote by $\phi_{\alpha,z}$, $\theta_{\alpha,z}$ its solutions satisfying the initial conditions
\begin{equation}\lb{ralpha}
\begin{bmatrix}
\phi_{\alpha, z}^{\1}(0)& \theta^{\1}_{\alpha, z}(0)\\ \phi_{\alpha, z}(0)& \theta_{\alpha, z}(0)
\end{bmatrix}
= R_{\alpha}^{-1}, \qquad R_{\alpha} :=
\begin{bmatrix}
\cos(\alpha)& -\sin(\alpha)\\ \sin(\alpha) & \cos(\alpha)
\end{bmatrix}.
\end{equation}
The solutions are entire with respect to $z$. In the special case $\alpha=0$, we denote $\theta_z:=\theta_{\alpha, z}, \phi_z:=\phi_{\alpha, z}$. Note that any $u\in \mathfrak D$ solving $\ell u=zu$ satisfies
\begin{equation}
\begin{bmatrix}
u^{\1}(x)\\
u(x)
\end{bmatrix}
=
T(z;x,0)
\begin{bmatrix}
	u^{\1}(0)\\
	u(0)
\end{bmatrix},\qquad T(z;x,0) :=\begin{bmatrix}
	\phi_z^{\1}(x)& \theta^{\1}_z(x)\\ \phi_z(x)& \theta_z(x)
\end{bmatrix}.
\end{equation}
Since $W(\phi_z, \theta_z)(x)$ is constant (due to Lagrange identity \cite[Lemma 2.3]{EGNT13}), $\det T(z;x,0)=1$. Thus, the transfer matrix can be defined as
\begin{equation}\lb{trmat}
T(z; x, y):=\begin{bmatrix}
	\phi_z^{\1}(x)& \theta^{\1}_z(x)\\ \phi_z(x)& \theta_z(x)
\end{bmatrix}\begin{bmatrix}
	\phi_z^{\1}(y)& \theta^{\1}_z(y)\\ \phi_z(y)& \theta_z(y)
\end{bmatrix}^{-1},
\end{equation}
where for any $u\in \mathfrak D$ solving $\ell u=zu$, and any $x, y\geq 0$,
\begin{equation}
	T(z;x,y) \begin{bmatrix}
		u^{\1}(y) \\ u(y)
	\end{bmatrix} = \begin{bmatrix}
		u^{\1}(x)\\ u(x)
	\end{bmatrix}.
\end{equation}
We will often denote $T(z;x):= T(z;x,0)$. 

Next, we turn to the Weyl-Titchmarsh theory for $\ell$. Assuming Hypothesis \ref{pot}, since $\ell$ is limit point at infinity, for any $z\in\bbC\setminus \bbR$, there is a 1-dim set of solutions in $L^2(\bbR_+)$ to $\ell u=zu$, where any such non-trivial solution is called a Weyl-Titchmarsh solution at infinity and denoted by $\psi_{\alpha,z}$. Fix any $\psi_{\alpha,z}$; the Weyl-Titchmarsh $m$-function is given by 
\begin{equation}\lb{mpsi}
m_{\alpha}(z)= -\frac{W(\psi_{\alpha,z},\theta_{\alpha,z})}{W(\psi_{\alpha,z},\phi_{\alpha,z})}=\frac{\cos(\alpha) \psi_{\alpha, z}^{\1}(0) - \sin(\alpha) \psi_{\alpha, z}(0)}{\sin(\alpha) \psi_{\alpha, z}^{\1}(0) + \cos(\alpha) \psi_{\alpha, z}(0)} ,
\end{equation}
where note that $m_{\alpha}(z)$ is independent of the choice of $\psi_{\alpha,z}$. Note that the boundary condition affects the Weyl function by a rotation matrix: denoting by $\simeq$ the projective relation on $\bbC^2 \setminus \{0\}$,
\begin{equation}\label{rotationonmalpha}
\begin{bmatrix}
m_\alpha(z) \\
1
\end{bmatrix}
\simeq
R_\alpha \begin{bmatrix}
m_0(z) \\
1
\end{bmatrix}.
\end{equation}

 Our next objective is to show that $ m_{\alpha}(z)$ can be obtained by the intersection of Weyl disks defined as
\begin{align}
\begin{split}\lb{dal}
D_x^{\alpha}(z) :&= \left\{\cU_{\alpha}\Big| u\neq 0\in \mathfrak D,\, \ell u-zu=0,\, \bfi W(\overline{u},u)(x)<0 \right\}, \\
\cU_{\alpha}:&= \frac{\cos(\alpha) u^{\1}(0) - \sin(\alpha) u(0)}{\sin(\alpha) u^{\1}(0) + \cos(\alpha) u(0)}.
\end{split}
\end{align}
To motivate this definition, let us reformulate it using the M\"obius transformations. To that end, denote $\hat{\mathbb{C}}=\mathbb{C}\cup\{\infty\}$ and introduce the quotient map $\pi:\mathbb{C}^2\setminus\{\vec{0}\} \to \hat{\mathbb{C}}$ given by $\pi(\begin{smallmatrix} w_1\\w_2\end{smallmatrix}) = \frac{w_1}{w_2}$. The M\"obius transformation $\cM[A]:\hat{\mathbb{C}}\to \hat{\mathbb{C}}$ associated with $A\in M_{2\times 2}(\mathbb{C})$ is uniquely defined via $\cM[A]\circ \pi = \pi \circ A$. Returning to \eqref{dal}, note that
\begin{align}
&R_{\alpha}\begin{bmatrix}
u^{\1}(0)\\u(0)
\end{bmatrix}
=
R_{\alpha}T^{-1}(z;x,0)
\begin{bmatrix}
u^{\1}(x)\\u(x)
\end{bmatrix},\\
&\bfi W(\overline u, u)(x)<0 \quad\iff\quad\frac{u^{\1}(x)}{u(x)}\in\bbC_+,
\end{align}
where $\bbC_+:=\{z\in\bbC: \im z>0\}$. Therefore,
\begin{equation}\lb{dal1}
D^{\alpha}_x(z)=\cM[R_{\alpha}T^{-1}(z;x,0)](\bbC_+).
\end{equation}
This identity, together with the fact that M\"obius transformations map generalized disks in $\hat{\bbC}$ to generalized disks, yields that $D^{\alpha}_x(z)$ is a generalized disk in $\hat{\bbC}$. We will show below that, for $z\in\bbC_+$ and $x>0$, the disks $D_x^{\alpha}(z)\subset \bbC_+$ shrink to a point in a monotone fashion as $x\uparrow +\infty$.

\begin{proposition}\lb{prop2.4}
	Let $z\in\mathbb{C}_+$, and consider $u\in\mathfrak D$, $\ell u=zu$. Then, for any $x\geq 0$,
	\begin{equation}\lb{firstidcl}
	2\Im z  \int_0^x |u(t)|^2\,\dd t = \bfi W(\overline{u},u)(x) - \bfi W(\overline{u},u)(0)\,. 
	\end{equation}
	In addition, if $u\not=0$, then the function 
	\begin{equation}
	\bfi W(\overline{u},u)(x) = -2\Im(\overline{u(x)}u^{\1}(x)) 
	\end{equation}
	is real-valued and strictly increasing in $x$.
\end{proposition}
\begin{proof}
	Using the identities $u' = u^{\1} + \sigma u$, $\overline{u}' = \overline{u}^{\1} + \sigma\overline{u}$, $ (u^{\1})' = -\sigma u^{\1} + (\tau-\sigma^2)u - zu$, and $(\overline{u}^{\1})' = -\sigma \overline{u}^{\1} + (\tau-\sigma^2)\overline{u} - \overline{z}\overline{u}$, we compute that \[ \bfi W(\overline{u},u)'  = \bfi (\overline{u}u^{\1} - \overline{u}^{\1}u)' = 2(\Im z)u\overline{u}\,. \] Then, integrating both sides of the above identity from $0$ to $x$ yields \eqref{firstidcl}. Next, since \[ \bfi W(\overline{u},u)(x) = -2 \Im (\overline{u(x)}u^{\1}(x))\,, \] $\bfi W(\overline{u},u)(x)$ is a real-valued function of $x$. If $u$ is a non-trivial eigensolution and $u(y) = 0$ for some $y$, then $u^{\1}(y)\neq 0$ and thus $u$ only has isolated zeros. In particular, $\bfi W(\overline{u},u)' = 2(\im z)u\overline{u}>0$ away from a discrete set, and so $\bfi W(\overline{u},u)$ is strictly increasing.
\end{proof}

The strict increasing property above corresponds to the fact that the operator obeys the Atkinson condition, or equivalently, the corresponding canonical system has no singular intervals. We now describe the Weyl disk formalism:

\begin{proposition}\lb{prop2.5}
	Assume Hypothesis \ref{pot} and fix $z\in\mathbb{C}_+$, $\alpha\in [0,\pi)$. Then,
	\begin{itemize}
		\item[(i)] For  $x> 0$, the set $D_x^{\alpha}(z)$ is a  disk in $\mathbb{C}_+$. 
		\item[(ii)]The disks from (i) are strictly nested; i.e.,
		\begin{equation}\lb{nest}
		\overline{D_y^{\alpha}(z)}\subset D_x^{\alpha}(z),\qquad x<y.
		\end{equation}
		\item[(iii)] The intersection of these disks is a single element set consisting of the Weyl-Titchmarsh coefficient $m_{\alpha}(z)$; i.e.,
		\begin{equation}\lb{inter}
		\bigcap_{x\geq 0}\overline{ D_x^{\alpha}(z)} = \{m_{\alpha}(z)\}.
		\end{equation}
		Moreover,  $\theta_{\alpha, z}+m_{\alpha}(z)\phi_{\alpha, z}$ is a Weyl-Titchmarsh solution. 
		\item[(iv)]  The mapping $z\mapsto m_{\alpha}(z)$ is a Herglotz function; i.e., analytic function $\bbC_+\to\bbC_+$.
	\end{itemize}
\end{proposition}
\begin{proof}
	{\it Parts (i) and (ii)}. Recall that M\"obius transformations map generalized disks in $\hat{\bbC}$ to generalized disks, and the boundary circles are mapped accordingly. To prove \eqref{nest}, fix $u\in\mathfrak D$ solving $\ell u=zu$ with $\cU_{\alpha}\in \overline{D_y^{\alpha}(z)}$, where recall $\cU_{\alpha}$ from \eqref{dal}. By definition of $\overline{D^{\alpha}(y)}$, $\bfi W(\overline{u},u)(y)\leq 0$, and by \eqref{firstidcl}, $\bfi W(\overline{u},u)(x)<0$. Thus, $\cU_{\alpha}\in {D_x^{\alpha}(z)}$. Finally, for any $x>0$, due to \eqref{dal1}, $D^{\alpha}_x(z)$ a generalized disk and, since $\overline{D_x^{\alpha}(z)}\subset D_0^{\alpha}(z)=\bbC_+$, it is a disk contained in $\bbC_+$.

	{\it Part (iii)}. Pick any $w\in \bigcap_{x\geq 0} \overline{D_x^{\alpha}(z)}$, where note that the intersection is not empty by parts (i), (ii). Our objective is to show that $w=m_{\alpha}(z)$. To that end, let $\psi_{\alpha,z}$ be the unique solution to the initial value problem
	\begin{equation}\lb{uww}
	\ell \psi_{\alpha,z}-z\psi_{\alpha,z}=0,\qquad \begin{bmatrix}
		\psi_{\alpha,z}^{\1}(0) \\ \psi_{\alpha,z}(0)
	\end{bmatrix} = \begin{bmatrix}
		\cos\alpha & \sin\alpha \\
		-\sin\alpha & \cos\alpha
	\end{bmatrix}\begin{bmatrix}
		w\\ 1
	\end{bmatrix} = R_{\alpha}^{-1}\begin{bmatrix}
		w\\ 1
	\end{bmatrix}.
	\end{equation}
	Let us see that $\psi_{\alpha,z}$ is a Weyl-Titchmarsh solution at infinity. Since
	\begin{equation}
		R_{\alpha}\begin{bmatrix}
		\psi_{\alpha,z}^{\1}(0) \\ \psi_{\alpha,z}(0)
	\end{bmatrix} = R_{\alpha}R_{\alpha}^{-1}\begin{bmatrix}
		w\\ 1
	\end{bmatrix} = \begin{bmatrix}
		w\\ 1
	\end{bmatrix}\,,
	\end{equation}
	applying the quotient map $\pi:\mathbb{C}^2\setminus\{\vec{0}\} \to \hat{\mathbb{C}}$ on both sides, it follows that
	\begin{equation}\lb{wm}
	w = \frac{\cos(\alpha) \psi_{\alpha,z}^{\1}(0) - \sin(\alpha) \psi_{\alpha,z}(0)}{\sin(\alpha) \psi_{\alpha,z}^{\1}(0) + \cos(\alpha) \psi_{\alpha,z}(0)}  \,.
	\end{equation}
	Since $w\in {D_x^{\alpha}(z)}$ for all $x> 0$, $\bfi W(\overline{\psi_{\alpha,z}},\psi_{\alpha,z})(x)<0$ for all $x>0$. Thus, by \eqref{firstidcl}, \[ 2\im z\int_0^x |\psi_{\alpha,z}(y)|^2\,\dd y = iW(\overline{\psi_{\alpha,z}},\psi_{\alpha,z})(x)-iW(\overline{\psi_{\alpha,z}},\psi_{\alpha,z})(0) \leq -iW(\overline{\psi_{\alpha,z}},\psi_{\alpha,z})(0), \]
	showing that $\psi_{\alpha,z}\in L^2(\bbR_+)$. Since $u_w$ is a non-trivial $L^2(\bbR_+)$ solution to $\ell u=zu$, it is a Weyl-Titchmarsh solution at infinity and thus
\begin{equation}
m_{\alpha}(z)=\frac{\cos(\alpha) \psi_{\alpha, z}^{\1}(0) - \sin(\alpha) \psi_{\alpha, z}(0)}{\sin(\alpha) \psi_{\alpha, z}^{\1}(0) + \cos(\alpha) \psi_{\alpha, z}(0)}  = w\,.
\end{equation}
Moreover, note that $W(\psi_{\alpha,z}, \phi_{\alpha,z})(0) =  1$:
\begin{align*}
	\psi_{\alpha,z}(0)\cos\alpha + \psi_{\alpha,z}^{\1}(0)\sin\alpha & = \begin{bmatrix}
		\sin\alpha & \cos\alpha
	\end{bmatrix} \begin{bmatrix}
		\psi_{\alpha,z}^{\1}(0)\\ \psi_{\alpha,z}(0)
	\end{bmatrix} = \begin{bmatrix}
		\sin\alpha & \cos\alpha
	\end{bmatrix} R_{\alpha}^{-1} \begin{bmatrix}
		w\\ 1
	\end{bmatrix} = 1\,;
\end{align*}
thus, $\psi_{\alpha,z} = \theta_{\alpha, z}+m_{\alpha}(z)\phi_{\alpha, z}$ is in fact the normalized Weyl-Titchmarsh solution at infinity.

{\it Part (iv).} Follows from \cite[Theorem 8.2]{EGNT13} and parts (i), (ii), (iii) above.
\end{proof}

To conclude this subsection, we recall from \cite[Section 9]{EGNT13} the spectral decomposition for the operator $H^{\alpha}$. The  Herglotz function $m_{\alpha}$ discussed in Proposition \ref{prop2.5} (iv) gives rise to a Borel measure $\mu^{\alpha}$ via the Stieltjes--Livsic inversion formula
\begin{equation}\lb{malp}
\mu^{\alpha}((\lambda_1, \lambda_2]):=\lim\limits_{\delta\downarrow 0}\lim\limits_{\varepsilon\downarrow 0}\frac{1}{\pi}\int^{\lambda_2+\delta}_{\lambda_1+\delta}\Im m_{\alpha}(\alpha+\bfi \varepsilon)d\lambda,
\end{equation}
for real numbers $\lambda_1<\lambda_2$. The operator $H^{\alpha}$ is unitarily equivalent to the operator of multiplication by the independent variable in the space $L^2(\bbR, \mu^{\alpha})$ and the classical spectral description via boundary values of $m_{\alpha}(z)$ holds, see \cite[Section 9]{EGNT13}. For instance, as in the classical setting, if $\alpha - \beta \notin \pi \bbZ$, then a.c.\ parts of $\mu^\alpha$, $\mu^\beta$ are mutually a.c., and their singular parts are mutually singular.

We will return to a detailed analysis of the absolutely continuous part of spectral measure $\mu^{\alpha}$ in Section \ref{sec2.3}, where we will rely on estimates for eigensolutions discussed next.

\subsection{Eigensolution estimates}\lb{eigensol} In this section, we derive auxiliary estimates for solutions of $\ell u=Eu$, $E\in\bbR$, $u\in\mathfrak D$. To describe the main assertions, let us fix $\lambda>0$ and denote
\[
\vec{u}(x) := \begin{pmatrix}
u^{\1}(x) \\
u(x)
\end{pmatrix}, \qquad \|\vec{u}(x)\|^2:=|u^{\1}(x)|^2+|u(x)|^2.
\]
In the estimates that follow, we give bounds with explicit dependence on the parameter $E$; we do not optimize these estimates, but we will use explicit estimates in some of the proofs that follow.

\begin{lemma}\label{lem2.6}
 Assume Hypothesis \ref{pot}. There exist constants $C_1, C_2, C_3, C_4, C_5 \in (0,\infty)$  which 
depend only on $\lVert \sigma \rVert_{2,\unif}$, $\lVert \tau \rVert$ such that, for every $E \in \bbR$ and every real-valued solution $u\in\mathfrak D$ of $\ell u = E u$:

(i) on every interval $I\subset\bbR_+$ with $|I|=\lambda \le 1$,
	\begin{align}
	\|\vec{u}(y)\|\leq 
	 C_1 e^{\lambda \lvert E \rvert} 
	 \|\vec{u}(x)\|,\qquad \forall x,y\in I\lb{comp1}.
	\end{align}

(ii) on every closed interval $I\subset\bbR_+$ with $|I|=\lambda \le 1$,
\[
 \max_{x \in I} \lvert u^{[1]}(x) \rvert \le C_2 \frac{ e^{\lambda \lvert E \rvert}}{\lambda} \max_{x \in I} \lvert u(x) \rvert.
\]

(iii) for $\delta = C_3(1 + |E|)^{-1}$,  at least one of the infimums
\[
\inf_{[y-\delta, y] \cap \bbR_+} \lvert u(x) \rvert , \qquad \inf_{[y, y+\delta]} \lvert u(x) \rvert 
\]
is larger or equal to $\lvert u(y) \rvert / 2$. 

(iv) for every $\epsilon \ge \frac 14$ and every $y \ge \epsilon$,
\begin{equation}\label{comp2}
\lvert u(y) \rvert^2 \le C_4 (1+\lvert E\rvert)^{2} \int_{y - \epsilon}^{y+\epsilon} \lvert u(x) \rvert^2 \,dx
\end{equation}

(v)  for every $\epsilon  \ge \frac 12$ and every $y \ge \epsilon$,
\[
\lvert u^{[1]}(y) \rvert^2 \le C_5 e^{2\epsilon |E|} (1+|E|)^2  \int_{y - \epsilon}^{y+\epsilon} \lvert u(x) \rvert^2 \,dx
\]
\end{lemma}

\begin{proof}
(i) From $\vec u ' = A \vec u$ for $x < y$ we obtain by Gronwall's inequality \cite[Lemma 1.3]{MR1004432}
\[
\lVert \vec u(y) \rVert \le e^{\int_x^y \lVert A(t,E) \rVert \,dt }  \lVert \vec u(x) \rVert.
\]
The operator norm bound
\begin{equation}\lb{conste1}
\|A(t,E)\|\leq 1 + 2|\sigma(t)| + |\tau(t)| + |\sigma^2(t)| + |E|
\end{equation}
implies that $\lVert A(t,E) \rVert$ is uniformly locally integrable: 
on every interval $I$ of length $|I| = \lambda \le 1$,
\[
\int_I \|A(t,E)\| dt \le 1 + 2 \lVert \sigma \rVert_{2,\unif} + \lVert \tau \rVert + \lVert \sigma \rVert_{2,\unif}^2 + \lambda |E|. 
\]
The case $y < x$ follows analogously.
	
(ii) We fix
\begin{equation}\lb{cnot}
S := \frac{ 1}{C_1 (3+ \lVert \sigma \rVert_{2,\unif}) } \lambda e^{-\lambda |E|}
\end{equation}
 and assume  that for some $y_0 \in I$,
\[
\lVert \vec u(y_0) \rVert \ge \lvert u^{[1]}(y_0) \rvert > \frac 1S \lvert u(x) \rvert  \qquad \forall x \in I.
\]
Combining, we conclude that for all $x,y \in I$,
\[
\lVert  \vec u(y) \rVert \ge \frac 1{C_1 e^{\lambda |E|}} \lVert \vec u(y_0) \rVert  > \frac 1{C_1 e^{\lambda |E|} S} \lvert u(x) \rvert.
\]
Since $C_1 e^{\lambda |E|} S < 1$, this implies
\[
\lvert u^{[1]}(y) \rvert > \left( \frac 1{C_1^2 e^{2\lambda |E|} S^2} - 1 \right)^{1/2} \lvert u(x) \rvert, \qquad \forall x,y\in I.
\]
In particular, $u^{\1}$ has no zeros on the interval $I$, so it has constant sign there. Thus,
		\begin{equation}\lb{ccn}
		 \left|\int_{I}u^{\1}(t)\,\dd t \right| = \int_{I} |u^{\1}(t)|\,\dd t > \lambda \left( \frac 1{C_1^2 e^{2\lambda |E|} S^2} - 1 \right)^{1/2} \max_{x \in I} \lvert u(x) \rvert.
		\end{equation}
		On the other hand, denoting the end points of $I$ by $j^-<j^+$, one has 
		\begin{align}\lb{signu}
		\begin{split}
		\left|\int_{I}u^{\1}(t)\,\dd t \right| & = \left| u(j^+) - u(j^-) - \int_{I}\sigma(t)u(t)\,\dd t \right| \\
		& \leq 2 \max_{x \in I} \lvert u(x) \rvert  + \sqrt{\lambda} \|\sigma\|_{2, \unif} \max_{x \in I} \lvert u(x) \rvert.
		\end{split}
		\end{align}
Since $u$ is not identically zero on $I$, combining \eqref{ccn} and \eqref{signu}, we obtain
		\begin{equation}
 \lambda \left( \frac 1{C_1^2 e^{2\lambda |E|} S^2} - 1 \right)^{1/2} < 2 + \|\sigma\|_{2,\unif}  \sqrt \lambda \le 2  + \|\sigma\|_{2,\unif} 
		\end{equation}
which implies
\[
\frac 1{C_1^2 e^{2\lambda |E|} S^2} < \frac {(2  + \|\sigma\|_{2,\unif})^2}{\lambda^2} + 1 \le \frac {(3  + \|\sigma\|_{2,\unif})^2}{\lambda^2} 
\]
and contradicts \eqref{cnot}.

(iii) Impose $C_3 \le 1$ to ensure $\delta \le 1$. 
Assume that the claim is false: then $u(y) \neq 0$ and by continuity there exist $x_1 \in [y-\delta, y] \cap [0,\infty)$ and $x_2 \in [y,y+\delta]$ such that $\lvert u(x_1) \rvert = \lvert u(x_2) \rvert = \lvert u(y) \rvert / 2$. In particular, $x_1 < y < x_2$. Pick $s \in [x_1, x_2]$ so that
\[
\lvert u(s) \rvert = \max_{x\in [x_1,x_2]} \lvert u(x) \rvert.
\]
By considering $\pm u$, without loss of generality we can assume $u(s) > 0$. 

Moreover, let us assume $u^{[1]}(s) \ge 0$ and work on the interval $[s,x_2]$; the other case is analogous by working on $[x_1,s]$.

The first step is an upper bound for the quasiderivative. For $x\in [s,x_2]$, denote
\[
h(x) = e^{\int_s^x \sigma(t) \,dt} u^{[1]}(x).
\]
Then the equation for $(u^{[1]})'$  implies
\[
h'(x) = e^{\int_s^x \sigma(t) \,dt} \left( \tau(x) - \sigma(x)^2 - E \right) u(x).
\]
Since $x - s \le  x_2 - s < 2\delta \le 1$, we use
\[
\left\lvert \int_s^x \sigma(t) \,dt \right\rvert \le \lvert x- s\rvert^{1/2} \lVert \sigma \rVert_{2,\unif}, \qquad \int_s^x \left\lvert \tau(t) - \sigma(t)^2 - E \right\rvert \dd t \le \lVert \sigma \rVert_{2,\unif}^2 + \lVert \tau\rVert + \lvert E \rvert
\]
to conclude that for $x\in [s, x_2]$, for some constant $C$,
\begin{align*}
| h(x)  - h(s)|  & \le \int_s^x e^{ C |t-s|^{1/2} } \lvert \tau(t) - \sigma(t)^2 - E \rvert \lvert u(t) \rvert \dd t \\
&  \le e^{ C |x-s|^{1/2} }  u(s) \int_s^x \lvert \tau(t) - \sigma(t)^2 - E \rvert  \dd t \\
&  \le ( C + \lvert E \rvert ) e^{ C |x-s|^{1/2} }  u(s).
\end{align*}
Since $h(s) = u^{[1]}(s) \ge 0$, we turn this into a one-sided bound 
\[
- h(x) \le - h(s) + ( C + \lvert E \rvert ) e^{ C |x-s|^{1/2} } u(s)  \le ( C + \lvert E \rvert ) e^{ C |x-s|^{1/2} } u(s) 
\]
and from this we finally obtain
\begin{equation}\label{quasiderivativeupperbound}
- u^{[1]}(x) \le ( C + \lvert E \rvert ) e^{2 C |x -s|^{1/2} } u(s), \qquad \forall x \in [s, x_2].
\end{equation}
Then we expand for $x \in [s,x_2]$,
\[
u(x) = u(s) + \int_s^x u'(t) \,dt = u(s) + \int_s^x u^{[1]}(t) \,dt + \int_s^x \sigma(t) u(t) \,dt 
\]
and by using \eqref{quasiderivativeupperbound} we get
\[
u(x) \ge  u(s) - |x-s| e^{2 C |x -s|^{1/2} } ( C + \lvert E \rvert ) u(s)    - |x-s| \lVert \sigma \rVert_{2,\unif} u(s).
\]
Plugging in $x = x_2$, recalling that $u(x_2) \le u(s)/ 2$ and $|x_2 - s| < 2\delta$, and dividing by $u(s)$ we obtain
\[
\frac 12 > 1 - 2 \delta e^{4 C \delta^{1/2} } ( C + \lvert E \rvert )   - \delta \lVert \sigma \rVert_{2,\unif}
\]
Equivalently,
\[
2 \delta e^{4 C \delta^{1/2} } ( C + \lvert E \rvert )  + \delta \lVert \sigma \rVert_{2,\unif} > \frac 12
\]
which gives a contradiction if $\delta$ is small enough.

(iv) It follows from (iii) that
\[
\int_{y-\epsilon}^{y+\epsilon} \lvert u(x) \rvert^2 \,dx \ge \frac {\lvert u(y) \rvert^2}4  \min \{ \epsilon, \delta\}. 
\]

(v) Without loss of generality assume $\epsilon \le 1$. Starting with (ii) and then (iv), with $a = \epsilon / 2$,
\[
\lvert u^{[1]}(y) \rvert^2 \le C_2^2 \frac {e^{4a|E|}}{(2a)^2} \max_{x \in [y-a,y+a]} \lvert u(x) \rvert^2  \le  C_2^2 C_4 \frac {e^{4a|E|}}{(2a)^2} (1+|E|)^2 \max_{x \in [y-a,y+a]} \int_{x-a}^{x+a} \lvert u(t) \rvert^2 \,dt
\]
which implies
\[
\lvert u^{[1]}(y) \rvert^2 \le C_5 e^{2\epsilon |E|} (1+|E|)^2  \int_{x-\epsilon}^{x+\epsilon} \lvert u(t) \rvert^2 \,dt. \qedhere
\]
\end{proof}

\begin{proof}[Proof of Theorem~\ref{theoremweightedL2eigenfunction}]
By considering $\Re u, \Im u$, it suffices to consider real-valued eigensolutions. Denote by $C$ the supremum in \eqref{goodweight}. By Lemma~\ref{lem2.6}, there exists $M$ such that
\begin{equation}\label{8jun1}
w(x) \lVert \vec u(x) \rVert^2 \le w(x) M \int_{x-1}^{x+1} u(y)^2 \,dy \le C M \int_{x-1}^{x+1} w(y) u(y)^2 \,dy.
\end{equation}
Integrating and using Tonelli's theorem gives
\begin{equation}\label{8jun2}
\int_1^l w(x) \lVert \vec u(x)\rVert^2 \dd x \le CM \int_1^l \int_{x-1}^{x+1} w(y) u(y)^2 \dd y \dd x \le 2 CM \int_0^{l+1} w(y) u(y)^2 \dd y.
\end{equation}
From now on assume $\int_0^\infty w(x) \lvert u(x)\rvert^2 \,dx < \infty$. Letting $l \to \infty$ in \eqref{8jun2} shows
\[
\int_1^\infty w(x) \lVert \vec u(x)\rVert^2 \dd x < \infty.
\]
By \eqref{goodweight}, $w$ is bounded on $(0,1)$, so $\int_0^1 w(x) u^{[1]}(x)^2 dx < \infty$. Using decaying tails of an integrable function, \eqref{8jun1} implies the pointwise decay \eqref{pointwisedecay}. 
\end{proof}

As a first application, we prove a Simon--Stolz  type criterion for absence of pure point spectrum, cf. \cite{MR1342046}.

\begin{lemma}\lb{lemmaSimonStolz}
Assume Hypothesis \ref{pot}. If for some $E \in \bbR$,
\begin{equation}\lb{infi}
	\int_0^{\infty} \frac{dx}{\|T(E, x)\|^2}=\infty,
\end{equation}
then $H^\alpha$ has no nontrivial solutions in $L^2((0,\infty))$; in particular, $H^\alpha$ doesn't have an eigenvalue at $E$ for any $\alpha \in [0,\pi)$.
\end{lemma}

\begin{proof}
Fix nontrivial $u\in\mathfrak D$, $\ell u= Eu$. By Theorem~\ref{theoremweightedL2eigenfunction} with $w=1$, it suffices to show that
\begin{equation}\lb{ll2}
\sqrt{(u^{\1})^2+u^2}\not \in L^2(\bbR_+). 
\end{equation} 
Since $T(E,x) \in \SL$ implies $\|T(E,x)\|=\|T(E,x)^{-1}\|$, 
\begin{equation}
\|\vec{u}(x)\|\geq C\frac{\|\vec{u}(0)\|}{\|T(E,x)\|},\qquad \vec{u}(x)=(u^{\1}(x), u(x))^{\top},
\end{equation}
which together with \eqref{infi} yields \eqref{ll2} as required. 
\end{proof}

\subsection{The absolutely continuous spectrum via Last--Simon approach}\lb{sec2.3} The main goal of this section is to develop the Last--Simon approach, cf. \cite{MR1666767}, to absolutely continuous spectrum via growth of transfer matrices.  To do this, we first discuss the relation between the subordinacy theory and the growth of transfer matrices. We say that $u\in \mathfrak D$ is a subordinate solution of $\ell u-zu=0$ if for some solution $\ell v-zv=0, v\in\mathfrak D\setminus\{0\}$, 
\begin{equation}\lb{sabor}
\lim\limits_{x\rightarrow\infty}\frac{\|u\|_{x}}{\|v\|_{x}}=0,\qquad \|f\|^2_{x}: = \int_0^x|f(y)|^2\,\dd y.
\end{equation}
Note that if \eqref{sabor} holds for some eigensolution $v$, it holds for every eigensolution linearly independent with $u$. Moreover, taking $v = \ol u$, we see that if a subordinate solution exists, it must be linearly dependent with its complex conjugate, so it must be a multiple of $\phi_{\alpha,z}$ for some $\alpha$.

For $\mu$-a.e.\ $\lambda \in \bbR$, the normal boundary value $\lim_{\epsilon \downarrow 0} m(\lambda+i\epsilon)$ exists in $\ol{\bbC_+}$. Subordinacy theory relates this value to the existence of subordinate solutions \cite{MR915965,MR1738043}; this was recently understood to be a special case of bulk universality in a general Hamiltonian system setting \cite{EichingerLukicSimanek}. To explain this, incorporate the boundary condition into the transfer matrix by defining
\[
T_\alpha(z;x) = R_\alpha  \begin{pmatrix}  \phi_{\alpha,z}^{[1]}(0) & \theta_{\alpha,z}^{[1]}(0) \\ 
\phi_{\alpha,z}(0) & \theta_{\alpha,z}(0)
\end{pmatrix}.
\]
This transfer matrix $T_\alpha(z;x)$ obeys the initial value problem
\[
\partial_x T_\alpha(z;x) = R_\alpha  A(z,x) R_\alpha^{-1} T_\alpha(z;x), \qquad T_\alpha(z;0) = I.
\]
This is a special case of a so-called Hamiltonian system, and can be written as
\[
j \partial_x T_\alpha(z;x) = 
R_\alpha
\begin{pmatrix} 1 &\sigma(x) \\ \sigma(x) & \sigma(x)^2 - \tau(x)+ z \end{pmatrix} 
R_\alpha^{-1} T_\alpha(z;x), \qquad j = \begin{pmatrix} 0 & -1 \\ 1 & 0 \end{pmatrix}.
\]
The transfer matrices generate 
a matrix kernel
\begin{align*}
\cK_l( z, w) & = \int_0^l T_\alpha(w;x)^* R_\alpha \begin{pmatrix} 0 & 0 \\ 0 & 1 \end{pmatrix} R_\alpha^{-1} T_\alpha(z;x) \,dx  \\
& = \int_0^l  \begin{pmatrix}  \phi_{\alpha,z}(x) \ol{ \phi_{\alpha,w}(x)} & \theta_{\alpha,z}(x) \ol{ \phi_{\alpha,w}(x)} \\
\phi_{\alpha,z}(x) \ol{ \theta_{\alpha,w}(x)} & \theta_{\alpha,z}(x) \ol{ \theta_{\alpha,w}(x)}
\end{pmatrix} \,dx.
\end{align*}
By the Cauchy--Schwarz inequality, the solution $\phi_{\alpha,z}$ is subordinate if and only if
\[
\lim_{l\to\infty} \frac 1{\tr \cK_l(E,E)} \cK_l(E,E) = \begin{pmatrix} 0 & 0 \\ 0 & 1 \end{pmatrix}.
\]
Scaling limits of $\cK_l$ are related to the normal limits of $m$-function:  by \cite[Theorem~1.8]{EichingerLukicSimanek},
\[
\lim_{\epsilon\downarrow 0} m_\alpha(E+i\epsilon) = \infty \quad\iff\quad \lim_{l\to\infty} \frac 1{\tr \cK_l(E,E)} \cK_l(E,E) = \begin{pmatrix} 0 & 0 \\ 0 & 1 \end{pmatrix}.
\]
Using \eqref{rotationonmalpha} to restate in terms of $m_0$, we conclude: 
 
\begin{lemma}\lb{lemma.subordinacy1} Assume Hypothesis \ref{pot}. For any $E \in \bbR$,
\[
\lim_{\epsilon\downarrow 0} m_0(E+i\epsilon) = -\cot \alpha  \quad\iff\quad \phi_{\alpha,E}\text{ is subordinate}.
\]
\end{lemma}

We also denote
\begin{equation}\lb{nset}
N(\ell):= \{E\in\bbR: \text{\ no solution of\ }\ell u-Eu=0\text{\ is subordinate}\}. 
\end{equation} 
Taking the union over $\alpha$ in Lemma~\ref{lemma.subordinacy1} and taking negations, for every $E$ for which the normal limit exists, $E \in N(\ell)$ if and only if
\[
\lim_{\epsilon\downarrow 0} m_0(E+i\epsilon)  \in \bbC_+.
\]
Recall that we denote by $\mu^\alpha_\ac$ the absolutely continuous part of the spectral measure $\mu^\alpha$.

\begin{lemma}\lb{thm2.10new} Assume Hypothesis \ref{pot}. For arbitrary $\alpha\in[0,\pi)$, $N(\ell)$ is an essential support for the absolutely continuous spectrum of $H^\alpha$ in the sense that $\mu^\alpha_\ac$ is mutually absolutely continuous with $\chi_{N(\ell)}(E) \,dE$. In particular,
	\begin{align}
	&\spec_{\ac}(H^{\alpha})=\overline{N(\ell)}^{\ess}. \lb{acclaim}
	\end{align}
\end{lemma}

\begin{proof}
Recall from \cite[Corollary 9.4]{EGNT13} that an essential support for $\mu^\alpha_\ac$ is the set
\begin{equation}
M_{\ac}:=\{E\in \bbR \mid 0<\limsup_{\epsilon\downarrow 0} \Im m_{\alpha}(\lambda+\bfi\varepsilon)<\infty\}. 
\end{equation}
Since $m_\alpha$ has a normal boundary value in $\ol{\bbC_+}$ for Lebesgue-a.e.\ $E$ (see e.g., \cite[Theorem 3.27, Corollary 3.29]{Tes14}), the set
\[
\{E\in \bbR \mid  \lim_{\epsilon\downarrow 0} m_{\alpha}(\lambda+\bfi\varepsilon) \in \bbC_+ \}
\]
is also an essential support for the a.c.\ spectrum. This set is independent of $\alpha$ by \eqref{rotationonmalpha}. By the observation proceeding the Lemma, the set $N(\ell)$ is another essential support for the a.c.\ spectrum of $H^\alpha$.
\end{proof}

\begin{proof}[Proof of Theorem~\ref{thm2.10}]
Since the spectral type of the a.c.\ part is independent of $\alpha$ (Lemma~\ref{thm2.10new}),  it suffices to prove the claim for $\alpha=0$. Assuming this value, we drop symbol $\alpha$ from subsequent notation.  

Due to preservation of Wronskian we have $\|\vec{\phi}_{\alpha, E}(x)\|\|\vec{\theta}_{\alpha, E}(x)\|\geq 1$; thus,
	\begin{equation}\lb{wronskint}
	(l-1)^2\leq {\|\vec{\phi}_{\alpha, E}\|^2_{L^2((1, l))}}{\|\vec{\theta}_{\alpha, E}\|^2_{L^2((1, l))}},\qquad l>1.
	\end{equation}
	Then, one has
	\begin{align}
	\frac{\|{\phi}_{\alpha, E}\|^2_{L^2((0, l+1))}}{\|{\theta}_{\alpha, E}\|^2_{L^2((0, l+1))}}&\underset{\eqref{derest1}}{\leq} C \frac{\|\vec{\phi}_{\alpha, E}\|^2_{L^2((0, l+1))}}{\|\vec{\theta}_{\alpha, E}\|^2_{L^2((1, l))}}\underset{\eqref{wronskint}}{\leq} C \frac{\|\vec{\phi}_{\alpha, E}\|^4_{L^2((0, l+1))}}{(l-1)^{2}}\leq C \left(\frac{\int_{0}^{l+1}\|T(E;x)\|^2dx}{l-1}\right)^2,
	\end{align}
where in the last step, we used  $\|\vec{\phi}_{\alpha, E}(x)\|\leq \|T(E,x)\|$. If the solution $\phi_{\alpha,E}$ is subordinate, taking the limit $l\to\infty$ shows
\[
\lim\limits_{l\rightarrow\infty}\frac{1}{l}\int_{0}^{l}\|T(E;x)\|^2dx=\infty.
\]
In other words, for the set $\Sigma_{\ac}$ defined by \eqref{sigmaac}, we conclude $\Sigma_{\ac}\subset N(\ell)$.

Therefore, to complete the proof, it is enough to show that
\begin{equation}
\liminf_{l\rightarrow\infty}\frac{1}{l}\int_{0}^{l}\|T(E;x)\|^2dx<\infty,\qquad \text{ for }\,\mu_{\ac}\text{-a.e. } E.
\end{equation}
To that end, let us fix $\gamma > 1$ and introduce the measure 
\begin{equation}\lb{fincl}
d\rho:=\frac{\min \{\mu^0, \mu ^{\frac \pi 2}\}}{e^{2\gamma |E|}}.
\end{equation}
Since $d\rho$ is equivalent to $\mu_{\ac}$, in order to prove that $\Sigma_\ac$ is an essential support for $\mu_{\ac}$, it is enough to show
\begin{equation}\lb{mainthm2.8}
\int_{\bbR}\left( \liminf_{l\rightarrow\infty}\frac{1}{l}\int_{0}^{l}\|T(E;x)\|^2dx\right)d\rho(E)<\infty.
\end{equation}
To that end, we will prove the following auxiliary inequalities: there exists   $\Upsilon>0$ such that for all $x\in(2,\infty)$,
\begin{align}
&\int_{\bbR} \frac{\int_{x-1}^{x+1}|\phi_E(t)|^2dt}{e^{\gamma |E|}}d\mu^0(E)<\Upsilon,\qquad \int_{\bbR} \frac{\int_{x-1}^{x+1}|\theta_E(t)|^2dt}{e^{\gamma |E|}}d\mu^{\frac \pi 2}(E)<\Upsilon,\lb{es1}\\
&\int_{\bbR} \frac{\int_{x-1}^{x+1}|\phi_E^{\1}(t)|^2dt}{e^{2\gamma |E|}}d\mu^0(E)<\Upsilon,\qquad \int_{\bbR} \frac{\int_{x-1}^{x+1}|\theta^{\1}_E(t)|^2dt}{e^{2\gamma |E|}}d\mu^{\frac\pi 2}(E)<\Upsilon.\lb{es2}
\end{align} 
We will prove the first parts of \eqref{es1}, \eqref{es2}, the second parts can be proved analogously. Since $\supp\mu^{0}$ is bounded from below, for some $\Lambda < \min \supp \mu^{0}$,
\begin{align}
\begin{split}
\int_{\bbR} \frac{|\phi_E(x)|^2}{e^{\gamma |E|}}d\mu^0(E)&\leq c \int_{\bbR} \frac{|\phi_E(x)|^2}{E-\Lambda}d\mu^0(E).
\end{split}
\end{align}
Then, using spectral representation of Green's function \cite[Lemma 9.6]{EGNT13} and the last part of Theorem \ref{thmoper}, we obtain
\begin{align}
\begin{split}
\int_{\bbR} \frac{|\phi_E(x)|^2}{e^{\gamma |E|}}d\mu^0(E)&\,\,\leq\,\, C \int_{\bbR} \frac{|\phi_E(x)|^2}{E-\Lambda}d\mu^0(E)\\
\quad=G(\Lambda; x,x)&\underset{\eqref{resineq}}{\leq} C G^{free}(\lambda(\sigma, \tau)-\Lambda; x,x)\leq \alpha(1-e^{-\beta|x|})\lb{es1.1},
\end{split}
\end{align}
where $\lambda(\sigma, \tau)$ is as in \eqref{resineq}, $G$ and $G^{free}$ denote respectively the Green's functions for $H^0$ and the free Dirichlet Laplacian on $\bbR_+$, i.e., for $\sigma=\tau=0$, and the constants $\alpha, \beta$ depend only on $\Lambda, \sigma, \tau$. Integrating \eqref{es1.1} yields the first inequality in \eqref{es1}.

Next, we switch to the first inequality in \eqref{es2}. By Lemma \ref{lem2.6},
\begin{equation}
|\phi_E^{\1}(t)|^{2}\leq C(E)\int_{t-1/2}^{t+1/2}|\phi_E(y)|^2dy,
\end{equation}
with $C(E)=\cO(e^{\gamma|E|}), E\rightarrow\infty$. Then, one has
\begin{equation}
\int_{\bbR} \frac{\int_{x-1}^{x+1}|\phi_E^{\1}(t)|^2dt}{e^{2\gamma |E|}}d\mu^0(E)\leq \int_{\bbR} C(E)\frac{\int_{x-3/2}^{x+3/2}|\phi_E(t)|^2dt}{e^{2\gamma |E|}}d\mu^0(E) < \Upsilon,
\end{equation}
where in the last step we used  \eqref{es1.1}. Next, \eqref{es1} and \eqref{es2} together yield a constant $C>0$ such that for all $x\in(2,\infty)$,
\begin{equation}\lb{axtest1}
\int_{\bbR}\int_{x-1}^{x+1}\|T(E;t)\|^2dt\,d\rho(E)<C.
\end{equation}
Splitting the interval $(0,l)$ into disjoint intervals of length 2, averaging over $l$, and applying Fatou's lemma gives
\[
\int_{\bbR} \liminf_{l \to \infty} \frac 1l \int_{0}^{l}\|T(E;t)\|^2dt\,d\rho(E) \le C
\]
which implies \eqref{mainthm2.8}.
\end{proof}

\begin{proof}[Proof of Theorem \ref{cor2.11}]
Using the estimate \eqref{axtest1} together with $\|T(E;t,s)\|\leq \|T(E;s)\| \|T(E;t)\|$ and the Cauchy--Schwarz inequality in $L^2(\bbR, d\rho)$ gives
\begin{equation}
\sup\limits_{x,y\in(2,\infty)}\int_{\bbR}\int_{x-1}^{x+1}\int_{y-1}^{y+1}\|T(E;t,s)\|^2dt\,ds\,d\rho(E)<\infty,
\end{equation}
which implies by Fatou's lemma that for $\rho$-a.e.\ $E$,
\[
\liminf_{j\to\infty} \int_{x_j-1}^{x_j+1}\int_{y_j-1}^{y_j+1}\|T(E;t,s)\|^2dt\,ds < \infty.
\]
By Lemma \ref{lem2.6}, for any $E$ there exists $C>0$ such that 
\begin{equation}
C^{-1} \int\limits_{x_j-1}^{x_j+1}\ \int\limits_{y_j-1}^{y_j+1}\|T(E;x,y)\|^2dxdy\leq \|T(E;x_j, y_j)\|^2\leq C \int\limits_{x_j-1}^{x_j+1}\ \int\limits_{y_j-1}^{y_j+1}\|T(E;x,y)\|^2dxdy,
\end{equation}
so for $\rho$-a.e.\ $E$,
\[
\liminf_{j\to\infty} \lVert T(E;x_j,y_j) \rVert < \infty. \qedhere
\]
\end{proof}

Theorem~\ref{cor2.11} will be our principal tool for showing the absence of absolutely continuous spectrum for a class of slowly decaying potentials, see Theorem \ref{sparseAltnewtheorem}(b).

\subsection{Carmona formula and pure a.c.\ spectrum on intervals}\lb{Carm}
In this section, we discuss a Carmona-type, cf. \cite{MR701057}, approximation result for the spectral measure of $H^{\alpha}$ and use it to derive a criterion for pure a.c. spectrum on an interval. This is our main tool for showing purely absolutely continuous spectrum for a class of slowly decaying potentials, see Theorem \ref{sparseAltnewtheorem}(a). 

\begin{theorem}\lb{carmona}
Assume Hypothesis \ref{pot}. For any $\alpha\in [0,\pi)$,
the measures
\begin{equation}\lb{meshr}
\dd\mu_x^{\alpha}(E) = \frac{1}{\pi ( \phi_{\alpha,E}(x)^2 + \phi^{[1]}_{\alpha,E}(x)^2 ) }\,\dd E,\qquad x> 0,
\end{equation}
converge vaguely to the spectral measure $\mu^{\alpha}$ of $H^{\alpha}$ as $x\to \infty$ in the sense that
\begin{equation}\lb{car}
\lim_{x\to\infty}\int\limits_{\bbR} h(E)\,\dd\mu_x^{\alpha}(E) = \int\limits_{\bbR}  h(E)\,\dd\mu^{\alpha}(E),\qquad \forall h\in C_0(\bbR). 
\end{equation}
\end{theorem}

\begin{proof}
Recall $R_{\alpha}$ from \eqref{ralpha}. For $z\in\bbC_+$ and $x>0$, let us define $m_{x,\alpha}(z)\in\bbC$ via 
	\begin{equation}\lb{mxa}
	\begin{bmatrix}
	m_{x,\alpha}(z) \\ 1
	\end{bmatrix} \simeq R_{\alpha}T(z;x)^{-1}\begin{bmatrix}
	\bfi \\1
	\end{bmatrix}.
	\end{equation} 
	In other words, $m_{x,\alpha}(z)$ is the image of $\bfi$ under the M\"obius transform $\cM[R_{\alpha}T(z;x,0)^{-1}]$. By Proposition \ref{prop2.5}, $\bfi \in \mathbb{C}_+$ implies $m_{x,\alpha}(z)\in D_x^{\alpha}(z)\subset \bbC_+$, and thus the function $z\mapsto m_{x,\alpha}(z)$ is Herglotz; moreover, since the disks $D_x^{\alpha}(z)$ shrink to a single point, for every $z\in\bbC_+$, one has $m_{x,\alpha}(z)\to m_{\alpha}(z)$ as $x\rightarrow\infty$. Our next objective is to compute boundary value of $\Im m_{x,\alpha}(E+\bfi \varepsilon)$ as $\varepsilon\downarrow 0$. Put
	\begin{align}
	&P(z, x):=\cos(\alpha) (\bfi\theta_{\alpha, z}(x)-\theta_{\alpha, z}^{\1}(x))+\sin(\alpha) (\bfi \phi_{\alpha, z}(x) - \phi_{\alpha, z}^{\1}(x)),\\
	&Q(z,x):=\sin(\alpha)(\bfi\theta_{\alpha, z}(x)-\theta_{\alpha, z}^{\1}(x)) + \cos(\alpha) (-\bfi\phi_{\alpha, z}(x)  + \phi_{\alpha, z}^{\1}(x)),
	\end{align}
	and rewrite \eqref{mxa} as 
	\begin{align}
	m_{x,\alpha}(z) & =\frac{P(z, x)}{Q(z, x)}=\frac{P(z, x)\overline{Q(z, x)}}{|Q(z, x)|^2}.
	\end{align}
	Note that both the denominator and the numerator are entire functions of $z$. Moreover, we claim that $|Q(z, x)|^2$ does not vanish for all $z\in \bbC_+\cup \bbR$ and $x>0$. Since $m_{x,\alpha}\in \bbC_+$ whenever $z\in \bbC_+$, it suffices to check the claim for $z\in \bbR$. Suppose for some $x>0$,
	\begin{equation}
	0=|Q(z, x)|^2= |(-\sin(\alpha)\theta^{\1}_{\alpha, z}(x) + \cos(\alpha)\phi^{\1}_{\alpha, z}(x)) + \bfi(\sin(\alpha)\theta_{\alpha, z}(x) -\cos(\alpha)\phi_{\alpha, z}(x))|^2.
	\end{equation}
Since $\phi_{\alpha, z}, \theta_{\alpha, z}, \phi_{\alpha, z}^{\1}, \theta_{\alpha, z}^{\1}\in\bbR$ for $z\in\bbR$, $|Q(z, x)|^2=0$ implies $\Re Q(z,x) = \Im Q(z,x) =0$. Writing this in matrix form gives the system
\[
\begin{bmatrix}
\phi^{\1}_{\alpha, z}(x) & \theta^{\1}_{\alpha, z}(x) \\
\phi_{\alpha, z}(x) & \theta_{\alpha, z}(x)
\end{bmatrix}
\begin{bmatrix}
\cos \alpha \\
- \sin \alpha
\end{bmatrix}
=
\begin{bmatrix}
0 \\
0
\end{bmatrix}
,
\]
which is a contradiction since the matrix is invertible. Thus, $Q(z,x) \neq 0$ for $z \in \bbR$ and $m_{x,\alpha}(z)$ has a continuous extension to $\bbR$. To summarize,
	\begin{equation}\lb{imlim}
	\lim\limits_{\varepsilon\downarrow 0}\Im m_{x, \alpha}(E+\bfi \varepsilon)= \frac{\Im[P(E, x)\overline{Q(E, x)}]}{|Q(E, x)|^2},
	\end{equation}
	where
	\begin{align}\lb{imlim2}
	\begin{split}
	\Im[P(E, x)\overline{Q(E, x)}]=
	& \cos^2\alpha(\phi^{\1}_{\alpha, E}(x)\theta_{\alpha, E}(x) - \theta^{\1}_{\alpha, E}(x)\phi_{\alpha, E}(x)) \\
	&+\sin^2\alpha (-\phi_{\alpha, E}(x)\theta^{\1}_{\alpha, E}(x) + \phi^{\1}_{\alpha, E}(x)\theta_{\alpha, E}(x))\\
	&=(\sin^2(\alpha)+\cos^2(\alpha))W(\theta_{\alpha, E}, \phi_{\alpha, E})=1,
	\end{split}
	\end{align} 
 and 
	\begin{align}\lb{imlim3}
	\begin{split}
	|Q(E, x)|^2=& \big(\cos(\alpha)\phi^{\1}_{\alpha, z}(x) - \sin(\alpha)\theta^{\1}_{\alpha, z}(x)\big)^2 + \big(\cos(\alpha)\phi_{\alpha, z}(x) - \sin(\alpha)\theta_{\alpha, z}(x)\big)^2\\
	= & \left\|\begin{bmatrix}
	\phi_E^{\1}(x) & \theta_E^{\1}(x)\\
	\phi_E(x) & \theta_E(x)
	\end{bmatrix} \begin{bmatrix}
	\cos(\alpha) \\ -\sin(\alpha)
	\end{bmatrix} \right\|^2 =\left\| \begin{bmatrix}
	\phi_{\alpha,E}^{\1}(x) \\
	\phi_{\alpha,E}(x) 
	\end{bmatrix} \right\|^2.
	\end{split}
	\end{align}
	It follows from \eqref{imlim}, \eqref{imlim2}, and \eqref{imlim3} that the measure corresponding to the Herglotz function $m_{x, \alpha}(z)$ is given by \eqref{meshr}. Moreover, since $m_{x, \alpha}(z)\to m_{\alpha}(z)$ as $x\rightarrow\infty$, by their Herglotz representations, the corresponding measures converge in the sense as asserted. 
\end{proof}

Having established \eqref{car}, the proof of Theorem~\ref{Lpbiggerthan2criterion} is identical to that of \cite[Theorem 3.7]{MR1666767}:

\begin{proof}[Proof of Theorem~\ref{Lpbiggerthan2criterion}]
Choose a sequence $x_n\rightarrow\infty$ such that
	\begin{equation}
	\lim_{n\rightarrow\infty}\int_{E_1}^{E_2}\|T(E;x_n)\|^pdE<\infty.
	\end{equation}
	Since $\det T(E;x) = 1$ and $\|v_{\alpha}\|=1$,
\[
\|T(E;x_n)v_{\alpha}\|\geq \|T^{-1}(E;x_n)\|^{-1}\|v_{\alpha}\| =  \|T(E;x_n)\|^{-1},
\]
and thus for $q=p/2$,
	\begin{equation}
	\sup_n\int_{E_1}^{E_2}\left(\frac{1}{\pi\|T(E;x_n)v_{\alpha}\|^2}\right)^q\,\dd E\leq \sup_n \int_{E_1}^{E_2}(\pi^{-1}{\|T(E;x_n)\|^2})^q\,\dd E<\infty. 
	\end{equation}
	Hence, by \cite[Lemma 3.8]{MR1666767}, the weak limit $d\mu^{\alpha}$ of measures $d\mu^{\alpha}_{x_n}$ is purely absolutely continuous on $(E_1, E_2)$.
\end{proof}

In the study of decaying potentials, a variant of Carmona's formula is useful:

\begin{theorem}\label{carmona2}
Assume Hypothesis 2.1. For any  $\alpha\in [0,\pi)$, the measures
\begin{equation}\label{carmona2b}
\dd\nu_x^{\alpha} (E) =  \frac{\chi_{(0,\infty)}(E) \sqrt E}{\pi ( E \phi_{\alpha,E}(x)^2 + \phi^{[1]}_{\alpha,E}(x)^2 ) }\,\dd E\,,\qquad x>0
\end{equation}
converge vaguely on $(0,\infty)$ to $\mu^{\alpha}$ as $x\to \infty$ in the sense that 
\[ \lim_{x\to\infty}\int_0^{\infty} h(E)\,\dd\nu_x^{\alpha}(E) = \int_0^{\infty} h(E)\,\dd\mu^{\alpha}(E)\,,\qquad \forall h\in C_c((0,\infty))\,.
\]
\end{theorem}

\begin{proof}
We use the branch of $\sqrt{-z}$ on $\bbC_+$ such that $\re  \sqrt{-z} >0$, $\im \sqrt{-z}  < 0$. With this choice of branch, $-\sqrt{-z}$ is a Herglotz function which continuously extends to $\overline{\mathbb{C}_+}$ with values $\chi_{(0,\infty)}(E)\sqrt{E}$ on $\mathbb{R}$.  For $z\in \mathbb{C}_+$, $x>0$, define $m_{x,\alpha}(z)$ via
	\[ \begin{bmatrix}
		m_{x,\alpha}(z) \\ 1
	\end{bmatrix} \simeq  R_{\alpha}T(z;x)^{-1}\begin{bmatrix}
		-\sqrt{-z}  \\1
	\end{bmatrix} \,.\]
	Since $-\sqrt{-z}$ is Herglotz, $m_{x,\alpha}(z)\in D_{x,\alpha}(z)\subset \mathbb{C}_+$ and is Herglotz as well. Moreover, since $D_x^{\alpha}(z)$ shrinks to a point as $x\to \infty$, $m_{x,\alpha}(z)\to m_{\alpha}(z)$ as $x\to \infty$. By arguments analogous to the proof of Theorem~\ref{carmona}, $\Im m_{x,\alpha}(z)$ has a continuous extension to $(0,\infty)$ with
	\[ \lim_{\epsilon\downarrow 0} \im m_{x,\alpha}(E+i\epsilon) =\frac{ \sqrt E}{ E \phi_{\alpha,E}(x)^2 + \phi^{[1]}_{\alpha,E}(x)^2  }.
 \]
It follows from above that the measure corresponding to $m_{x,\alpha}(z)$ has the restriction to $(0,\infty)$ given by \eqref{carmona2b}, which concludes the proof.
\end{proof}

\subsection{Pr\"ufer variables}\lb{pruf} We now introduce Pr\"ufer variables associated with real eigensolutions of $\ell$ and relate their growth to that of the transfer matrices. In the locally integrable setting, Pr\"ufer variables are a well-established tool for spectral analysis for decaying potentials; we will use them in the proof of Theorem \ref{sparseAltnewtheorem}.

For $k>0$, consider the eigenvalue equation $\ell u=k^2u, u\in\mathfrak D$. For a nontrivial real-valued solution $u$, introduce $\theta:\bbR\rightarrow \bbR$, $R:\bbR\rightarrow(0,\infty)$ via the relations
\begin{align}\lb{uu1}
u(x)  = R(x) \sin(\theta(x)),\qquad u^{\1}(x) & = kR(x)\cos (\theta(x)).
\end{align}
Since a composition of a Lipschitz function with an absolutely continuous function is absolutely continuous, this can be done so that $R, \theta \in \AC([0,\infty))$. The remaining nonuniqueness in the choice of $\theta$ is usually fixed by setting $\theta(0) \in [0,2\pi)$.

\begin{proposition}\lb{prop2.4new} Assume Hypothesis \ref{pot}. For $k>0$, in terms of Pr\"ufer variables, the eigenfunction equation $\ell u = k^2u$ is equivalent to the system
	\begin{align}
	\label{ODE1}
	&\theta' = k - \frac{\tau -\sigma^2}{k} \sin^2(\theta) + \sigma\sin(2\theta), \\
	\label{ODE2}
	&(\log R)'  = \frac{\tau-\sigma^2}{2k}\sin(2\theta) - \sigma\cos (2\theta).
	\end{align}
Moreover, for any $\alpha,\beta\in(0,\infty)$, $\theta_1, \theta_2\in[0, 2\pi)$, there is a constant $C=C(\alpha, \beta, \theta_1, \theta_2)>1$ such that for all $k\in(\sqrt{\alpha}, \sqrt{\beta})$,
	\begin{equation}\lb{RTbound}
	\frac 1 C \max(R(x, \theta_1), R(x, \theta_2))\leq \|T(k^2;x)\|\leq C \max(R(x, \theta_1), R(x, \theta_2)). 
	\end{equation}
\end{proposition}
\begin{proof} First, we rewrite $\ell u-k^2 u=0$ as 
	\begin{equation}
	\begin{bmatrix}
	u^{\1}\\
	u
	\end{bmatrix}' =\begin{bmatrix}
	-\sigma & (\tau - \sigma^2) - k^2\\
	1 & \sigma
	\end{bmatrix}\begin{bmatrix}
	u^{\1} \\
	u
	\end{bmatrix},\qquad ':=\frac{d}{dx}.
	\end{equation}
	Then, substituting \eqref{uu1} into the above equation, we obtain
	\begin{equation}\lb{dereq}
	\begin{bmatrix}
	kR'\cos (\theta)-R\theta'\sin (\theta)\\
	R'\sin(\theta)+R\theta'\cos(\theta)
	\end{bmatrix} =\begin{bmatrix}
	-\sigma k R\cos(\theta) + R(\tau - \sigma^2)\sin(\theta) - k^2R\sin(\theta)\\
	kR\cos(\theta)+ \sigma R\sin(\theta)\end{bmatrix}.
	\end{equation}
	To derive \eqref{ODE1}, we take the scalar product of  both sides of \eqref{dereq} and $(-\sin(\theta), k \cos(\theta))$; and to derive \eqref{ODE2}, we take the scalar product of sides of \eqref{dereq} and $(\cos(\theta), k\sin(\theta))$.  
	
	Let $u_1, u_2$ be solutions corresponding to the initial conditions $\theta(0)=\theta_1, \theta(0)=\theta_2$ respectively. Then, 
	\begin{equation}
	\cU(x)=
	T(k^2;x)
	\cU(0),\qquad \cU(x):=\begin{bmatrix}
	u^{[1]}_1(x) & u^{[1]}_2(x)\\
	u_1(x) & u_2(x)
	\end{bmatrix}.\end{equation}
	Using the representation \eqref{uu1}, one obtains 
	\begin{equation}\lb{c1c2}
	C_1\max(R(x, \theta_1), R(x, \theta_2))\leq \|\cU(x)\|\leq C_2\max(R(x, \theta_1), R(x, \theta_2)),\qquad x\geq 0
	\end{equation} 
	for some constants $C_1, C_2>0$ depending only on $\alpha, \beta$. Finally, since $T(k^2;x) \in\SL$, there exists constants $C_3, C_4>0$ depending only on $\theta_1, \theta_2$ such that
	\begin{equation}\lb{c3c4}
	C_3\|\cU(x)\| \leq \|T(k^2;x)\|\leq C_4 \| \cU(x)\|,\qquad x\geq 0.
	\end{equation}
	Combining \eqref{c1c2}, \eqref{c3c4} yields \eqref{RTbound}.
\end{proof}

\begin{proof}[Proof of Theorem~\ref{theorem.shortrange}]
Consider Pr\"ufer variables $R(x,E)$ associated to the solution $\phi_{\alpha,E}$ for $E \in (0,\infty)$. By \eqref{ODE2}, $\log R(x,E)$ converges uniformly as $x \to\infty$ on every compact interval $[E_1,E_2] \subset (0,\infty)$.  Recall $\dd\nu_x^{\alpha} (E)$ from Theorem \ref{carmona2}, then for any $h \in C_c((0,\infty))$, by uniform convergence,
\begin{equation}
\lim_{x\to\infty}  \int h(E) \dd\nu_x^{\alpha} (E) = 
\lim_{x\to\infty} \int  \frac{h(E) \sqrt E}{\pi ( E \phi_{\alpha,E}(x)^2 + \phi^{[1]}_{\alpha,E}(x)^2 ) }\,\dd E=\int \frac{ h(E) \sqrt E}{\pi ( \lim_{x\to\infty} R(x,E) )^{2}  }\,\dd E.
\end{equation}
Thus,
\[
\chi_{(0,\infty)}(E) \dd \mu^\alpha(E) =  \frac{ \chi_{(0,\infty)}(E) \sqrt E } { \pi ( \lim_{x\to\infty} R(x,E) )^{2} } \dd E. \qedhere
\]
\end{proof}

\section{Distributional sparse potentials. Investigation of spectral types}

In this section we prove Theorem \ref{sparseAltnewtheorem}.

\subsection{Decomposition of sparse potentials} The first step in the proof of Theorem \ref{sparseAltnewtheorem} is to reformulate it in terms of the Hryniv--Mykytyuk decomposition in a way that is consistent with the sparse structure of the potential.  If we applied their decomposition directly to $V$, the dependence on integers in \cite{HrMyk01} would complicate matters; instead, note that \cite[Lemma 2.2]{HrMyk01} gives a decomposition of $W_n \in H^{-1}(\bbR)$ with $\supp W_n \subset [-\Delta, \Delta]$ as $W_n = S_n' + T_n$ with $S_n \in L^2(\bbR)$, $T_n \in L^1(\bbR)$ supported in the same interval  (the authors use $\Delta = 1$ but this is merely a matter of rescaling). Moreover, this decomposition is continuous in $H^{-1}(\bbR)$-norm. Thus, we obtain
\begin{equation}\lb{wnw}
W_n=S_n'+T_n,\quad W=S'+T,
\end{equation} 
with 
\begin{align}
\begin{split}\lb{tnsn}
&\supp(S_n)\cup\supp(S)\cup \supp(T_n)\cup\supp(T)\subset[-\Delta, \Delta], \quad S'+T\not=0,\\
&S, S_n\in L^2(\bbR), T, T_n\in L^1(\bbR),\ \|S_n-S\|_{L^2(\bbR)}\rightarrow 0,\  \|T_n-T\|_{L^1(\bbR)}\rightarrow0.
\end{split}
\end{align}
In addition, without loss of generality, we can assume that $S \neq 0$ and $T \neq 0$: this is because if one of $S, T$ is identically equal to zero, we can pick arbitrary $h\in W^{1,1}(\bbR)$, $\supp(h)\subset[-\Delta, \Delta]$, $S+h\not \equiv 0$, $T-h'\not \equiv 0$. Notice that $W_n=(S_n+h)'+T_n-h'$, $W=(S+h)'+T-h'$.

In summary, we will use the following setup throughout this section:

\begin{hypothesis}\lb{sparsepot}
	Let $\{x_n\}_{n=1}^{\infty}\subset \bbR_+$ be a monotonically increasing sequence such that $x_1 > \Delta$ and
	\begin{equation}\lb{beta}
	\lim\limits_{n\rightarrow\infty}\frac{x_n}{x_{n+1}}=0. 
	\end{equation}
	Let  $\beta>1$ be so that $x_n\geq C\beta^n$ for a fixed constant $C>0$. Let  $T, S, T_n, S_n$ be as in \eqref{tnsn} and suppose, in addition,  $T\not\equiv 0$,  $S\not\equiv 0$.  Furthermore,  fix a sequence $\{d_n\}_{n=1}^{\infty}\subset \bbR$ with
	\begin{equation}\lb{dn}
	\lim_{n\to\infty} d_n = 0.
	\end{equation}
	Let sparse coefficients $\tau, \sigma$  be given by
	\begin{align}\lb{tausigma}
	&\tau(x) = \sum_{n=1}^{\infty} d_nT_n(x-x_n),\quad \sigma(x) = \sum_{n=1}^{\infty}d_nS_n(x-x_n).
	\end{align} 
	Fix arbitrary $\alpha\in[0,\pi)$, and let $H^{\alpha}$ be the corresponding Schr\"odinger operator as defined in Theorem \ref{thmoper}.
\end{hypothesis}

The rest of this paper is dedicated to the proof of Theorem \ref{sparseAltnewtheorem}; parts (a), (b) are provided  in Sections \ref{secac}, see page \pageref{pagea}, and \ref{secsc}, see page \pageref{pageb}, respectively.

\subsection{Auxiliary estimates for Pr\"ufer variables} We begin with a series of auxiliary results. The first one concerns estimates for Pr\"ufer variables and their $k$-derivatives near $x_n$ for large $n$. 

To streamline the exposition, in the remaining part of the paper, we will use $C$ for positive constants that vary from one inequality to the other but always remain $n$-independent. Also, whenever an inequality involving $n$ is mentioned without a specified range of admissible values of $n$, it is assumed that the range is $n\geq n_0$ for some $n_0$.
\begin{lemma}
	\label{argANDradEst}
	Assume Hypothesis \ref{sparsepot} and fix any compact interval $[E_1,E_2]\subset(0,\infty)$. Then, there exists a constant $C>0$ such that for all $k\in [\sqrt{E_1},\sqrt{E_2}]$ and sufficiently large $n$,
		\begin{align}
		&(i)\, \left| \frac{\partial \theta}{\partial k}(x_n+\Delta) \right| \leq Cx_n,\qquad  (ii)\, \left| \frac{\partial^2 \theta}{\partial k^2}(x_n+\Delta) \right| \leq C\min\left\{x_n^2, 1 + \sum_{m=1}^nd_mx_m^2 \right\},\lb{thetone}\qquad\\
		& (iii)\, \left|\log R(x_n+\Delta)\right| \leq C\sum_{m=1}^nd_m,\qquad  (iv)\, \left|\frac{\partial \log R}{\partial k} (x_n+\Delta)\right|\leq C\sum_{m=1}^nd_m x_m\,.\lb{logrn}
		\end{align}
\end{lemma}

\begin{proof}
	Proof of \eqref{thetone}(i). Fix any compact interval $[a,b]\subset\bbR$, $f,g\in L^{1}([a,b])$, and suppose $h\in \AC([a,b])$ satisfying $ h'(x) = f(x) + g(x)h(x)$. Then, for any $x\in [a,b]$,
	\begin{align}\lb{auxin2}
	|h(x)| & \leq |h(a)|e^{\int_a^b|g|\,\dd y} + \int_a^b |f| e^{\int_a^b|g|\,\dd t}\,\dd y  = (|h(a)| + \|f\|_{1,[a,b]}) e^{\|g\|_{1,[a,b]}}.
	\end{align} 
	Let $h(x):= \frac{\partial\theta}{\partial k}(x)$. Differentiating \eqref{ODE1} with respect to $k$, we have $\frac{\partial h}{\partial x}=f+gh$ with 
	\begin{align}\lb{fng}
	f(x,k):=1 + \frac{\tau-\sigma^2}{k^2}\sin^2(\theta),\qquad g(x,k):=  \sigma\cos(2\theta)- \frac{\tau-\sigma^2}{k}\sin(2\theta),
	\end{align} 
	which, for $[a,b]:=[x_n-\Delta, x_n+\Delta]$ and sufficiently large $n$, satisfy
	\begin{align}\lb{fgest}
	\|f(\cdot, k)\|_{L^1(a,b)}\leq 2\Delta+Cd_n,\qquad \|g(\cdot, k)\|_{L^1(a,b)}\leq Cd_n. 
	\end{align}
	Our objective is to prove that there exists $C>0$ such that for sufficiently large $n$,
	\begin{equation}\lb{hnxn}
	|h(x_n+y)| \leq Cx_n,\qquad y\in[-\Delta, \Delta].
	\end{equation}
	Note that $h'(x)=1$ for $x\in (x_{n-1}+\Delta, x_n-\Delta)$; thus,
	\begin{equation}\lb{hest}
	|h(x_{n}-\Delta)|\leq |h(x_{n-1}+\Delta)|+x_n-x_{n-1}-2\Delta.
	\end{equation}
	Then, using \eqref{auxin2} with $f,g$ as in \eqref{fng}, $[a,b]=[x_{n}-\Delta, x_n+y]$,  and employing \eqref{fgest}, 
	\begin{align}
\label{boundonH}
\begin{split}
|h(x_n+y)| & \,\,\,\leq\,\,\, \left(|h(x_{n}-\Delta)| + 2\Delta + Cd_n\right)e^{Cd_n}\\
& \underset{\eqref{hest}}{\leq }\left(|h(x_{n-1}+\Delta)| + x_n-x_{n-1} + Cd_n\right)e^{Cd_n}.
\end{split}
\end{align}
for sufficiently large $n$. Since $\beta>1$, $d_n\to 0$, and $x_n\to \infty$, there exists  $n_1\in\bbN$ such that for all $n\geq n_1$, 
	\begin{equation}\lb{boundonH2}
		 \beta^{-1}e^{Cd_n} \leq \frac{1+\beta^{-1}}{2}\qquad \text{and}\qquad \left(1 + \frac{Cd_n}{x_n}\right)e^{Cd_n} \leq 1+ \left(\frac{1-\beta^{-1}}{2}\right).
	\end{equation}
	For such $n_1$, let $D\geq 2$ be such that \[ |h(x_{n_1-1} + \Delta)|\leq Dx_{n_1-1}.\] We claim that for all $n\geq n_1$, 
	\begin{equation}\lb{xnd1}
	 |h(x_n+y)| \leq Dx_n,\qquad y\in[-\Delta, \Delta].
	\end{equation}  
	Indeed, using \eqref{beta} together with \eqref{boundonH} and \eqref{boundonH2}, for $n\geq n_1, y\in[-\Delta, \Delta]$, we have
	\begin{align*}
	|h(x_n+y)| & \leq (|h(x_{n-1}+\Delta)| + (x_n-x_{n-1}) + Cd_n)e^{Cd_n} \\
	& \leq ((D-1)x_{n-1} + x_n + Cd_n) e^{d_n}\\
	& \leq x_n\left((D-1)\beta^{-1} + 1 + \frac{Cd_n}{x_n} \right) e^{Cd_n}\\
	& \leq x_n \left(D- (D-2)\frac{1-\beta^{-1}}{2} \right) \leq D x_n,
	\end{align*}
   which yields \eqref{thetone}(i). 	
   
  Proof of \eqref{thetone}(ii). Let $w:= \frac{\partial h}{\partial k}= \frac{\partial^2\theta}{\partial k^2}$ and differentiate \eqref{ODE2} twice with respect to $k$, then 
	\begin{equation}\lb{dwx}
		\frac{\partial w}{\partial x} = F(x) + G(x)w(x),
	\end{equation}
	where $F(x):=F_1(x)+ F_2(x)h(x) + F_3(x)[h(x)]^2$ and
	\begin{align}
	F_1(x) & := -2\cdot \frac{\tau-\sigma^2}{k^3}\sin^2(\theta(x)),\qquad  F_2(x) := 2\cdot \frac{\tau-\sigma^2}{k^2}\sin(2\theta(x)),\\
	F_3(x) & := -2\cdot \frac{\tau-\sigma^2}{k}\cos(2\theta(x)) - 2\sigma\sin(2\theta(x))\\
	 G(x)&:= -\frac{\tau-\sigma^2}{k}\sin(2\theta(x)) + \sigma\cos(2\theta(x)).
	\end{align} 
	Note that for $[a,b]:=[x_n-\Delta, x_n+\Delta]$ and sufficiently large $n$,
	\begin{equation}\lb{FGin}
	\|F\|_{L^1(a,b)}\leq Cd_nx_n^2,\qquad \|G\|_{L^1(a,b)}\leq Cd_n,
	\end{equation}
	where we used \eqref{xnd1} in the first inequality. Then, using \eqref{auxin2} with $[a,b]=[x_{n}-\Delta, x_n+y]$, $f=F$, $g=G$, $w(x_n-\Delta)=w(x_{n-1}+\Delta)$, and \eqref{FGin}, 
	\begin{equation}\lb{wxny}
		|w(x_n+y)| \leq (|w(x_{n-1}+\Delta)| + Cd_nx_n^2)e^{Cd_n}.
	\end{equation}
	Since $\beta>1$ and $d_n\rightarrow0$,  for any $C>0$, there is large enough $n_2$ such that for all $n\geq n_2$,
	\begin{equation}\lb{bdn}
	\beta^{-2}e^{Cd_n}\leq \frac{1+\beta^{-1}}{2},\qquad Cd_ne^{Cd_n} \leq \frac{1-\beta^{-1}}{2}\,.
	\end{equation} 
	For such $n_2$, let $\wti {D}\geq 2$ be such that 
	\[ |w(x_{n_2-1}+\Delta)|\leq \wti {D}x_{n_2-1}^2\,. \] 
	We claim that  for all $n\geq n_2$, 
	\begin{equation}\lb{tildd}
	|w(x_n+y)|\leq \wti {D}x_n^2,\qquad y\in[-\Delta, \Delta].
	\end{equation}
	Proceed with induction in $n$: suppose \eqref{tildd} holds for $n-1$; then, employing \eqref{wxny}, for all $y\in[-\Delta, \Delta]$,
	\begin{align}
	\begin{split}\lb{par1}
		|w(x_n+y)| & \,\,\,\leq\,\,\, (|w(x_{n-1}+\Delta)| + Cd_nx_n^2)e^{Cd_n} \\
	& \,\,\,\leq\,\,\, (\wti {D}x_{n-1}^2 + Cd_nx_n^2)e^{Cd_n} \leq x_n^2(\wti{D}\beta^{-2} + Cd_n)e^{Cd_n}\\
	& \underset{\eqref{bdn}}{\leq} x_n^2\left( \wti{D}\frac{1+\beta^{-1}}{2} + \frac{1-\beta^{-1}}{2} \right) \leq x_n^2\left( \wti{D}+(1-\wti D) \frac{1-\beta^{-1}}{2} \right)\leq  \tilde{D}x_n^2.\qquad\quad
	\end{split}
	\end{align} 
To complete the proof of \eqref{thetone}(ii), integrate \eqref{dwx} over $[x_n-\Delta, x_n+\Delta]$ and use \eqref{tildd}; then, 
	\begin{align}
	|w(x_n+\Delta) -w(x_n-\Delta)|  \leq  Cd_n x_n^2,\qquad  n\geq n_2.
	\end{align} 
	Hence,
	\begin{align}
	\begin{split}\lb{par2}
		|w(x_n+\Delta)| & \leq |w(x_{n_2-1}+\Delta)| +  \sum_{m=n_2}^n|w(x_m+\Delta)-w(x_m-\Delta)|\\
	& \leq Dx_{n_2-1}^2 + C\sum_{m=n_2}^n d_mx_m^2\leq C\left(1+\sum_{m=1}^nd_mx_m^2 \right)\,.
	\end{split}
	\end{align}
Combining this with \eqref{par1} concludes the proof for \eqref{thetone}. 

Proof of \eqref{logrn}(iii) \& (iv). Note that $R(x_{n-1}+\Delta)=R(x_n-\Delta)$ as $\frac{\partial \log R}{\partial x}\equiv 0$ on $[x_{n-1}+\Delta,x_n-\Delta]$ . Integrating \eqref{ODE2} over $[x_n-\Delta, x_n+\Delta]$ and noting that $d_n\rightarrow0$, 
	\begin{align*}
	|\log R(x_n+\Delta)-\log R(x_n-\Delta)| & \leq Cd_n,
	\end{align*} 
which implies \eqref{logrn}(iii). To prove \eqref{logrn}(iv), differentiate \eqref{ODE2} with respect to $k$ to get
\begin{align}
\frac{\partial}{\partial x}\frac{\partial  \log R}{\partial k} = -\frac{\tau-\sigma^2}{2k^2}\sin(2\theta) + \frac{\tau-\sigma^2}{2k}\cos(2\theta)\frac{\partial \theta}{\partial k} + \sigma\sin(2\theta)\frac{\partial \theta}{\partial k};
\end{align}
then, integrate both sides over $[x_n-\Delta, x_n+\Delta]$ while noting \eqref{hnxn} and $d_n\rightarrow0$, 
	\begin{align}
	\left|\frac{\partial }{\partial k}\log R(x_n+\Delta) - \frac{\partial }{\partial k}\log R(x_n-\Delta) \right| \leq C d_nx_n.
	\end{align} 
	The latter, in turn, yields  \eqref{logrn}(iv).
\end{proof}

\begin{remark}
Lemma \ref{argANDradEst} and its proof are similar to \cite[Propositions 5.1, 5.2]{MR1628290}, where the case of $S_n=0$ and $T_n=T\in L^{\infty}(\bbR)$ was considered. We extend that proof to the case $S_n\not=0$, and $T_n\in L^1(\bbR)$ by using \eqref{auxin2}, which is an $L^1$ version of the key inequality \cite[eq. (5.7)]{MR1628290}, and verifying new inequalities \eqref{fgest}, \eqref{FGin}. 
\end{remark}
To streamline the exposition, we introduce the following notation
\begin{align}\lb{qndef}
q_n(y, k):= \frac{d_nT(y)-d_n^2S^2(y)}{2k},\qquad  \sigma_n(y):=d_nS(y).
\end{align}
Note that due to \eqref{tnsn}, for a fixed interval  $[\alpha,\beta]\subset (0,\infty)$, we have
\begin{equation}\lb{qnest}
\int_{-\Delta}^{\Delta}|q_n(y, k)|+|\sigma_n(y)| dy \underset{n\rightarrow\infty}{=}\cO(d_n),
\end{equation}
uniformly for  $k\in[\alpha,\beta]$. 

In the following Lemma, we provide the second order expansion of variable $\theta$ with respect to $d_n$ as $n\rightarrow\infty$. This result will be used in Lemma \ref{lem3.4} and the proof of Theorem \ref{sparseAltnewtheorem}(a). 
\begin{lemma}\lb{thetalem} Assume Hypothesis \ref{sparsepot} and fix any compact interval $[E_1,E_2]\subset(0,\infty)$. Then, the following asymptotic expansion 
	\begin{equation}\lb{eqlem33}
	\theta(x_n+y)\underset{n\rightarrow\infty}{=}\theta_n^{(0)}(y)+d_n\theta_n^{(1)}(y)+ \cO(d_n^2)
	\end{equation}
	holds uniformly for $y\in[-\Delta, \Delta]$, $k\in[\sqrt{E_1},\sqrt{E_2}]$, where, recalling \eqref{qndef}, 
	\begin{align}
	&{\theta}^{(0)}_n(y) := \theta(x_{n-1}+\Delta) + k(x_n + y - x_{n-1} - \Delta),\lb{thetzero}\\
	&{\theta}^{(1)}_n(y) := \frac{1}{d_n}\int_{-\Delta}^y\sigma_n(s)\sin(2{\theta}_n^{(0)}(s))-2q_n(s)\sin^2({\theta}_n^{(0)}(s))\,\dd s.\lb{thetonenew}
	\end{align}
\end{lemma}
\begin{proof}
	By \eqref{ODE1}, 
	\begin{align}
	\begin{split}\lb{frstord}
	&\left|\theta(x_n + y) - \theta_n^{(0)}(y)\right|\leq \int_{x_n-\Delta}^{x_n+y}|\theta'(s)- k |\,\dd s\\
	& = \int_{-\Delta}^y|\sigma_n(s)\sin(2{\theta}_n(x_n+s))-2q_n(s)\sin^2({\theta}(x_n+s))| \dd s\underset{n\rightarrow\infty}{=}\cO(d_n),
	\end{split}
	\end{align}
	where in the last step we used \eqref{qnest}. The argument for the second-order asymptotic formula is similar. Note that
	\begin{align}
	&|\theta(x_n+y)-\theta_n^{(0)}(y)-d_n\theta_n^{(1)}(y)|=\left|\int_{x_n-\Delta}^{x_n+y}(\theta'(s)- k )\,\dd s -d_n{\theta}^{(1)}_n(y)\right|\no\\
	&\,\,\,\leq \left| \int_{-\Delta}^y\sigma_n(s)[\sin(2{\theta}_n^{(0)}(s))-\sin(2{\theta}_n^{(0)}(x_n+s))] -2q_n(s)[\sin^2({\theta}_n^{(0)}(s))-\sin^2({\theta}(x_n+s))]\dd s\right|,\no\\
	&\underset{n\rightarrow\infty}{=}\cO(d_n^2),
	\end{align}
	where we used \eqref{qnest} and \eqref{frstord} in the last step.
\end{proof}
\begin{remark}
	A version of Lemma \ref{thetalem} with $S_n=0$ and $T_n=T\in L^{\infty}(-\Delta, \Delta)$ is discussed in \cite[Sections 5,6]{MR1628290}. In our case, notice that when $S_n\not=0$, the integral on the right-hand side of \eqref{thetonenew} contains an additional term $\sigma_n(s)\sin(2{\theta}_n^{(0)}(s))$. This will become relevant in the proof of Theorem \ref{sparseAltnewtheorem}(b). 
\end{remark}

\begin{corollary}
Assume the setting of Lemma \ref{thetalem}. Then, 
	\begin{equation}\lb{thetatilde}
	\frac{\partial\theta^{(0)}_n(y)}{\partial k} >\frac{x_n}{2},
	\end{equation}
	holds for sufficiently large $n$ and all $k\in[\sqrt{E_1},\sqrt{E_2}]$, $y\in[-\Delta, \Delta]$.
\end{corollary}
\begin{proof}
Differentiating \eqref{thetzero} with respect to $k$ and using \eqref{thetone}, we get
\begin{align}
\frac{\partial\theta^{(0)}_n(y)}{\partial k}  &:= \frac{\partial\theta(x_{n-1}+\Delta)}{\partial k}  + x_n + y - x_{n-1} - \Delta\\
&\,\,\geq x_n +y-(C+1)x_{n-1}+y>\frac{x_n}{2},
\end{align}
where we used \eqref{beta} in the last step. 
\end{proof}

To conclude this section, we show that \eqref{dn} rules out point spectrum for $H^{\alpha}$.
\begin{proposition}\lb{prop3.9}
Assume Hypothesis~\ref{sparsepot}. Then, $\spec_{\pp}(H^{\alpha})\cap (0,\infty)=\emptyset$ for all $\alpha\in[0,\pi)$. 
\end{proposition}

\begin{proof}
Consider the Pr\"ufer variables corresponding to a nontrivial real eigensolution $u$ at $E > 0$, normalized so that $R(0) = 1$. By \eqref{logrn}(iii),
\[
 R(x_n+\Delta)^2 \ge \exp \left( - 2 C \sum_{m=1}^{n}d_m \right).
\]
This means at most exponential decay of the sequence $R(x_n+ \Delta)^2$, since the sequence $d_n$ is bounded. Due to the superexponential growth \eqref{beta}, this implies
\[
(x_{n+1} -x_n - 2\Delta) R(x_n+\Delta)^2 \rightarrow \infty, \qquad n\rightarrow\infty.
\]
Since $R(x)$ is constant on $[x_n+\Delta, x_{n+1} - \Delta]$, this implies
\[
\int_0^\infty R(x)^2 \dd x = \infty
\]
and, by Theorem~\ref{theoremweightedL2eigenfunction}, this implies $u \notin L^2(\bbR_+)$.
\end{proof}

\subsection{Purely absolutely continuous spectrum}\lb{secac}
In this section, we provide the proof of Theorem \ref{sparseAltnewtheorem} part (a). 
\begin{proof}[Proof of Theorem \ref{sparseAltnewtheorem} $($a$)$] \lb{pagea}By Lemma \ref{lem2.4} $\spec_{\ess}(H^{\alpha})=[0,\infty)$. Then, by Theorem \ref{carmona}, it suffices to show that for every finite interval $[E_1,E_2]\subset (0,\infty)$,
\begin{equation}\lb{e1e2}
\liminf_{n\rightarrow\infty}\int\limits_{E_1}^{E_2}\|T(E;x_n+\Delta)\|^4\,\dd E<\infty.
\end{equation}	
In fact, we will show that for any $\theta\in[0,2\pi)$ and any non-negative $g\in C^{\infty}_0(0,\infty)$ (after possibly passing to a subsequence),
\begin{equation}\lb{supbn}
\sup_n B_n<\infty,\qquad B_n:=\int\limits_{0}^{\infty}g(k) |R(x_n+\Delta, \theta)|^4 \dd k<\infty.
\end{equation}
The latter together with Proposition \ref{prop2.4new} yields \eqref{e1e2}. Explicitly, we will derive a recursive inequality 
\begin{equation}\lb{bns}
B_n\leq (1+\rho_n)B_{n-1},
\end{equation}
for a sequence $\{\rho_n\}\in \ell^1(\bbN)$, $\rho_n>0$, which is sufficient for  \eqref{supbn}.  To that end, we integrate (\ref{ODE2}) over the interval $[x_n-\Delta, x_n+\Delta]$ and use $R(k;x_n-\Delta)=R(k;x_{n-1}+\Delta)$ to obtain
\begin{equation}\lb{r4}
R(k;x_n+\Delta)^4 = R(k;x_{n-1}-\Delta)^4\exp(Q_n),
\end{equation}
where
\begin{align*}
Q_n & = \frac{2}{k}\int_{-\Delta}^{\Delta} (d_nT_n(y) - b_n^2S^2(y))	\sin(2\theta(x_n+y))\,\dd y  - 4\int_{-\Delta}^{\Delta}d_nS_n(y)\cos(2\theta(x_n+y))\,\dd y.
\end{align*} 
Then,
\begin{align}\lb{qqtilde}
|Q_n- \wti {Q}_n| \leq C d_n^2
\end{align}
where 
\begin{equation}\lb{qntilde}
 \wti {Q}_n:=\frac{2}{k}\int_{-\Delta}^{\Delta} (d_nT_n(y) - d_n^2S_n^2(y))\sin(2\theta^{(0)}_n(y))\,\dd y - 4 \int_{-\Delta}^{\Delta}d_nS_n(y)\cos(2\theta^{(0)}_n(y))\,\dd y\,. 
\end{equation} 
and $\theta^{(0)}_n$ is as in  \eqref{thetzero}. Indeed, \eqref{qqtilde} follows readily from 
\begin{equation}
|\sin(2{\theta}_n(y))-\sin(2\theta^{(0)}_n(y))|\leq C |{\theta}_n(y)-\theta^{(0)}_n(y)|\leq C d_n^2,\qquad  y\in[-\Delta, \Delta].
\end{equation}

Returning back to \eqref{r4}, notice that \eqref{qqtilde} together with $|Q_n|\leq C d_n$ yields
\begin{align}
R(k;x_n+\Delta)^4   &\leq R(k;x_{n-1}-\Delta)^4 (1 + |Q_n|  + C Q_n^2)\\
&\leq R(k;x_{n-1}-\Delta)^4  \big(1 + |\wti Q_n| +Cd_n^2\big).
\end{align}	
To obtain \eqref{bns}, multiply the above inequalities by $g(k)$ and integrate over $(0,\infty)$; then, 
	\begin{align}
B_n  \leq B_{n-1}(1+Cd_n^2) + E_n,\lb{bnbn1}\qquad E_n:= \int g(k)R(k;x_{n-1} + \Delta)^4\wti {Q}_n\,\dd k.
\end{align} 
Recalling \eqref{qntilde} and exchanging the order of integration, we obtain
	\begin{align}
	\begin{split}\lb{3ens}
		E_n & = \int_{-\Delta}^{\Delta} 2d_nT_n(y)\int \frac{g(k)}{k} R(k;x_{n-1}+\Delta)^4\sin(2\theta^{(0)}_n(y))\,\dd k \,\dd y \\
	& \quad  -\int_{-\Delta}^{\Delta}2d_n^2S_n^2(y)\int \frac{g(k)}{k} R(k;x_{n-1}+\Delta)^4\sin(2\theta^{(0)}_n(y))\,\dd k \,\dd y \\
	& \quad -\int_{-\Delta}^{\Delta}4d_nS_n(y)\int g(k) R(k;x_{n-1}+\Delta)^4\cos(2\theta^{(0)}_n(y))\,\dd k \,\dd y\,,
	\end{split}
	\end{align}
where note that all terms above are of the form
\begin{equation}\lb{cen}
\cE_n:=\int_{-\Delta}^{\Delta} \gamma_n w_n(y)\int \Psi(k) R(k;x_{n-1}+\Delta)^4 u(2\theta^{(0)}_n(y))\,\dd k \,\dd y , 
\end{equation}
with  
\begin{align}
\begin{split}\lb{cen2}
& \gamma_n\in\{d_n, d_n^2\},\ \Psi\in C_0^{\infty}(\bbR_+),\quad u\in \{\sin(x), \cos(x)\}\\
&w_n\in\{T_n, S_n, S_n^2\}\subset L^{1}(\bbR), \ \sup_{n\geq 1}\|w_n\|_{L^1(\bbR)}<\infty.
\end{split}
\end{align}

{\it Claim.} For $B_{n}$ and $\cE_n$ defined in \eqref{supbn} and \eqref{cen} respectively, and $\beta>1$ as in Hypothesis \ref{sparsepot}, there is a sequence $\{s_n\}_{n\geq 1}\in\ell^1(\bbN)$ such that 
\begin{equation}\lb{auxin}
\cE_n\leq C\left(\beta^{-n/2}+B_{n-1}s_n\right).
\end{equation}
\begin{proof}[Proof of Claim] Let $v$ be either $\sin$ or $\cos$ so that one has  $u= v'$, and rewrite $\cE_n$ as 
\begin{equation}
\cE_n=\int_{-\Delta}^{\Delta} \gamma_nw_n(y)\int \Psi(k) R(k;x_{n-1}+\Delta)^4 \frac{1}{2\frac{\partial \theta^{(0)}_n}{\partial k}(y)} \frac{\partial v}{\partial k}(2\theta^{(0)}_n(y))\,\dd k \,\dd y.
\end{equation}
Next, we integrate by parts with respect to $k$ to obtain three integrals, each corresponding to applying $\partial_k$ to one of the three functions in 
\begin{equation}\lb{prod}
\Psi(k)\,\cdot\, R(k;x_{n-1}+\Delta)^4 \,\cdot\,\frac{1}{2\frac{\partial \theta^{(0)}_n}{\partial k}(y)}.
\end{equation}

{\it Case 1:} $\partial_k$ lands on the first term in \eqref{prod}. Then,
\begin{align}
&\left|\int_{-\Delta}^{\Delta} \gamma_n w_n(y)\int \frac{\partial \Psi(k)}{\partial k}\ R(k;x_{n-1}+\Delta)^4 \frac{v(2\theta^{(0)}_n(y))}{2\frac{\partial \theta^{(0)}_n}{\partial k}(y)}  \,\dd k \,\dd y\right|\\
&\quad \leq \frac{C d_n}{x_n} \exp\left(\sum_{m=1}^{n-1}d_m\right)\leq C \beta^{-n}d_n \exp\left(\frac{n\log\beta}{2}\right)\leq C\beta^{-n/2},
\end{align}
where  $\beta>1$ is such that $x_n\geq C\beta^n$ (Hypothesis \ref{sparsepot}) and in the first inequality, we used
 \begin{align}
 \log(R(k;x_{n-1}+\Delta))&\leq C\sum_{m=1}^{n-1}d_m \qquad\text{ by  }\eqref{logrn}(iii),\\
\begin{split}
\frac{1}{2\frac{\partial \theta^{(0)}_n}{\partial k}(y)}&\leq \frac{C}{x_n}\qquad\,\,\,\,\qquad\text{\ by\ }\eqref{thetatilde},\,\, |v(2\theta^{(0)}_n(y))|\leq 1,\\
\text{and }\quad \int_{-\Delta}^{\Delta} \int w_n(y)\Psi(k)dy\,dk &\leq C\qquad\quad\qquad\,\text{\ by\ }\eqref{cen2} \lb{vsink},
\end{split}
\end{align}
in the second inequality, we used $d_n=o(1)$.

{\it Case 2:} $\partial_k$ lands on the middle term in \eqref{prod}. We employ $\partial_kR^4=R^4\partial_k \log R^4$ and \eqref{logrn}\emph{(iv)} to estimate the $R$-term, and \eqref{thetatilde} to estimate the $\theta^{(0)}_n$-term as
\begin{align}
&\left|\int_{-\Delta}^{\Delta} \gamma_n w_n(y)\int\Psi(k)R(k;x_{n-1}+\Delta)^4 \partial_k\log (R(k;x_{n-1}+\Delta)^4)\frac{v(2\theta^{(0)}_n(y))}{2\frac{\partial \theta^{(0)}_n}{\partial k}(y)} \dd k \,\dd y\right|\\
&\quad \leq C d_n\left(\sum_{m=1}^{n-1}d_mx_m\right)\frac{1}{x_n}B_{n-1}\leq C  \eta_n B_{n-1},
\end{align}
where in the first inequality, we used \eqref{vsink} and
\begin{align}
& R(k;x_{n-1}+\Delta)^4\leq C\sum_{m=1}^{n-1}d_mx_m,\qquad \text{ by}\,\,\eqref{logrn} (iv);
\end{align}
in the second inequality, we used $d_m=o(1)$, \eqref{beta}, and 
\begin{equation}\lb{etan}
\eta_n:= \frac{d_n}{x_n} \sum_{m=1}^{n-1}d_mx_m,\qquad \text{with}\,\, \sum_{n=1}^{\infty}\eta_n<\infty,
\end{equation}
with \eqref{etan} proved in Remark \ref{propap12}. 

{\it Case 3:} $\partial_k$ lands on the last term in \eqref{prod}. In this case, we have
\begin{align}
&\left|\int_{-\Delta}^{\Delta} \gamma_n w_n(y)\int\Psi(k)R(k;x_{n-1}+\Delta)^4 \frac{1}{\left(\frac{\partial \theta^{(0)}_n}{\partial k}(y)\right)^2}\frac{\partial^2 \theta^{(0)}_n(y)}{\partial^2 k}v(2\theta^{(0)}_n(y)) \dd k \,\dd y\right|\\
&\quad \leq C d_n\left(1+\sum_{m=1}^{n-1}d_mx_m^2\right)\frac{1}{x_n^2}B_{n-1}\leq C\kappa_nB_{n-1},
\end{align}
where in the first inequality, we used \eqref{thetone}(ii) and \eqref{vsink}; in the second inequality, we set
\begin{equation}\lb{kapan}
\kappa_n:=\frac{d_n}{x_n^2}\left(1+\sum_{m=1}^{n-1}d_mx_m^2\right),\qquad \text{with}\,\, \sum_{n=1}^{\infty}\kappa_n<\infty,
\end{equation}
with \eqref{kapan} proved in Remark \ref{propap12}. 

Combining cases 1-3, we obtain \eqref{auxin}.
\end{proof}
Since all three terms on the right-hand side of \eqref{3ens} are of the type $\cE_n$,
\begin{equation}
E_n\leq C\left(\beta^{-n/2}+B_{n-1}s_n\right),\qquad \text{with}\,\, \{s_n=\eta_n+\kappa_n\}_{n\in\bbN}\in \ell^1(\bbN)
\end{equation}
Combining this with \eqref{bnbn1}, for sufficiently large $n$, 
\begin{equation}
B_n\leq B_{n-1}(1+Cd_n^2+C s_n) + C \beta^{-n/2}.
\end{equation}
Therefore, $\max(1, B_n)\leq (1+Cd_n^2+ s_n +  \beta^{-n/2}))\max(1, B_{n-1})$ and thus \eqref{supbn} holds. 
\end{proof}

\begin{remark}\lb{propap12} In the setting of Theorem \ref{sparseAltnewtheorem}(a), the numerical series introduced in \eqref{etan} and \eqref{kapan} are convergent due to \cite[Lemma 5.3]{MR1628290}; we expand the concise proof provided therein. For a numerical sequence $d=\{d_n\}_{n\in\bbN}$, consider the convolution operator
\begin{equation}
(T_{\gamma}d)_n:=\sum_{m=1}^{\infty}\gamma^{-|m-n|}d_m,\qquad \text{for}\,\,\gamma>1.
\end{equation}
By Young's inequality, $T_{\gamma}$ is a bounded linear operator on $\ell^2(\bbN)$. Let $\gamma>1$ be such that $\frac{x_m}{x_n}\leq C\gamma^{-|m-n|}$, $m\leq n$. Then, \eqref{etan} follows from
\begin{equation}\lb{sumdns}
 \sum_{n=1}^{\infty}d_n \sum_{m=1}^{n-1}d_m\frac{x_m}{x_n}\leq  \sum_{n=1}^{\infty}d_n \sum_{m=1}^{\infty}d_m\gamma^{-|m-n|}=\langle d, T_\gamma d \rangle_{\ell^2}\leq \|T_{\gamma}\|_{\cB(\ell^2(\bbN))}\|d\|^2_{\ell^2(\bbN)}<\infty,
\end{equation}
and \eqref{kapan} follows from 
\begin{align}
 \sum_{n=1}^{\infty}d_n\left(1+\sum_{m=1}^{n-1}d_m\left(\frac{x_m}{x_n}\right)^2\right)&\leq  \sum_{n=1}^{\infty}\left(C\beta^{-2n}+d_n\sum_{m=1}^{n-1}d_m\gamma^{-2|m-n|}\right)\\
 &\leq C+ \|T_{\gamma^2}\|_{\cB(\ell^2(\bbN))}\|d\|_{\ell^2(\bbN)}<\infty,
\end{align}
where $\beta$ is as in Hypothesis \ref{sparsepot} in the second inequality.
\end{remark}

\subsection{Purely singular continuous spectrum}\lb{secsc}
In this section, we provide the proof of Theorem \ref{sparseAltnewtheorem} part (b).

Since Proposition \ref{prop3.9} rules out the presence of positive eigenvalues, to demonstrate the absence of absolutely continuous spectrum, the strategy is to verify the conditions of Theorem \ref{cor2.11} via \eqref{RTbound} and 
\begin{equation}
\lim\limits_{j\rightarrow\infty}R(x_{n_j}+\Delta, k)=\infty.
\end{equation}

We begin with a set of auxiliary results concerning the Fourier transform of the potential. We will use the notation

\begin{equation}\lb{arz}
\hatt f (k):=\int_{-\infty}^{\infty}e^{2\bfi k y}f(y) dy,\quad  \Ar(\hatt f(z)):=\begin{cases}\arg(\hatt f(z)), &\hatt f(z)\not=0, \\
	0,  &\hatt f(z)=0,
\end{cases}
\end{equation}

\begin{lemma}\lb{lem312}
Assume Hypothesis \ref{sparsepot}. Then

(i) For $j=0,1,2$ one has 
\begin{align}\lb{arts}
	\frac{d^{j}}{dz^j}\Ar(\hatt T_n(z))\rightarrow  \frac{d^{j}}{dz^j}\Ar(\hatt T(z)),\ n\rightarrow\infty,
\end{align} 
uniformly for $z$ in compact intervals $\cI\subset (0,\infty)$ that contain no roots of $\hatt T$. In particular, for such $\cI$ one has
\begin{align}\lb{arts2}
\limsup\limits_{n\rightarrow \infty}\ \sup_{z\in \cI}	\left|\frac{d^{j}}{dz^j}\Ar(\hatt T_n(z))\right|<\infty,\  j=0,1,2.
\end{align}
Identical assertions hold with $T$ replaced by $S$. 

(ii) Let $\Phi(z):=(2z)^{-1}\hatt T(z)-\bfi \hatt S(z)$ and suppose that a compact interval $J\subset(0,\infty)$ contains no roots of $\Phi$. Then one has
\begin{align}\lb{newcondition}
\liminf\limits_{n\rightarrow \infty}\ \inf_{z\in J}	\left|\frac{\hatt T_n(z)}{2z}-\bfi \hatt S_n(z)\right|^2>0.
\end{align}
\end{lemma}
\begin{proof}
(i) Denote for simplicity $f_n(z):=\hatt T_n(z)$, $f(z):=\hatt T(z)$. Clearly, $f$ is entire function which is not identically zero and $f_n$ converges to $f$ uniformly on compacts. We claim that there exists $n_0\in\bbN$ such that:
	\begin{enumerate}
		\item for all $n \ge n_0$ and $z\in \cI$, $f_n(z) \neq 0$, 
		\item for $j=0, 1,2$, $\arg f_n^{(j)} \to \arg f^{(j)}$ uniformly on $\cI$ and in particular,
		\[
		\sup_{n\ge n_0} \sup_{z\in J} \lvert \arg f_n^{(j)}(z) \rvert < \infty.
		\]
	\end{enumerate}
To prove these two basics facts from complex analysis, first, recall that if $f_n \to f$ uniformly on some compact $K$, then for any compact $K' \subset \Int K$, $f_n' \to f'$ uniformly on $K'$; this holds by Cauchy's differentiation formula
\[
f_n'(z) = \frac 1{2\pi i} \oint_{\lvert w-z\rvert = \epsilon} \frac{f_n(w)}{(w-z)^2} \,dw
\]
applied with $\epsilon = \dist(K', \bbC \setminus K)$. Next, denote $D = \{ z \in\bbC \mid f(z) = 0 \}$,  $d = \dist(D, \cI) > 0$, and  $\cI_\epsilon := \{ z\in\bbC \mid \dist(z,\cI ) \le \epsilon \}$. 

On the set $\cI_{d/2}$, $f_n$ converge uniformly to $f$, so there exists $n_0$ such that for all $n \ge n_0$ and $z\in \cI_{d/2}$, $f_n(z) \neq 0$. By the above argument, $f_n' \to f'$ uniformly on $\cI_{d/3}$. Thus, $(\log f_n)' = f_n' / f_n \to f' / f = (\log f)'$ uniformly on $\cI_{d/3}$. Thus, $(\log f_n)'' \to (\log f)''$ uniformly on $\cI_{d/4}$. Taking imaginary parts we conclude $\arg f_n' \to \arg f'$ and $\arg f_n'' \to \arg f''$ uniformly on $\cI$.  Choosing branches so that $\log f_n(\min\cI) \to \log f(\min\cI)$ and taking limits of
\[
\log f_n(x) = \log f_n(\min\cI) + \int_{\min\cI}^x \frac{f_n'(y)}{f_n(y)} \dd y
\]
and taking imaginary parts shows uniform convergence of $\arg f_n$ to $\arg f$ on $\cI$.

(ii) The proof follows directly from complex analytic facts (1), (2) stated above with $f(z):=\hatt T(z)-2z \bfi  \hatt S(z) $, $f_n(z):=\hatt T_n(z)-2z \bfi  \hatt S_n(z) $. 
\end{proof}

Assuming Hypothesis \ref{sparsepot}, we say that a compact interval $J\subset \bbR_+:=(0,\infty)$ is  $(S, T)-$admissible if $J$ avoids zeros of  $\hatt S(z)$, $\hatt T(z)$, and $(2z)^{-1}\hatt T(z)-\bfi \hatt S(z)$, that is, 
\begin{align}
&J\cap \{ z\in(0,\infty): \hatt T(z)=0 \text{ or }\hatt S(z)=0 \text{ or } (2z)^{-1}\hatt T(z)-\bfi \hatt S(z)=0\}=\emptyset.
\end{align}

 In the following Lemma, we derive a third order expansion for the increment of $\log R(x_n + \Delta,k)$ with respect to $d_n$. For $\{z_n\}_{n\geq 0}\subset \bbC$ we denote  $\delta z_n:=z_n-z_{n-1}$. 
 
\begin{lemma}\lb{lem3.4}Assume Hypothesis \ref{sparsepot}, fix a finite interval $[E_1, E_2]\subset (0,\infty)$ such that $[\sqrt{E_1}, \sqrt{E_2}]$ is $(S,T)-$admissible, and define $Y_n(k) := \log R(x_n + \Delta,k)$. Then the following asymptotic expansion holds uniformly for $k\in[\sqrt{E_1}, \sqrt{E_2}]$ 
\begin{align}
	&\delta Y_n(k)\underset{n\rightarrow\infty}{=}X_n(k)+\wti X_n(k)+\overset{\circ}{X}_n(k)+\cO(d^3_n),
\end{align}	
where the oscillatory terms $X_n, \wti X_n$  and are given by 
\begin{align}
	X_n(k):=&\,\,d_n \int_{-\Delta}^{\Delta} \left[\frac{T_n(y)}{2k}\right]\sin(2{\theta}_n^{(0)}(y))-\left[\frac{d_nT_n(y)}{4k^2}\int_{-\Delta}^yT_n(s)\dd s \right]\cos(2{\theta}_n^{(0)}(y))\,\dd y\lb{i1}\qquad\\
	&\,\,-d_n\int_{-\Delta}^{\Delta}\left[\frac{d_nS_n(y)}{2k}\int_{-\Delta}^y T_n(s)\,\dd s\right]\sin(2{\theta}_n^{(0)}(y)) +S_n(y)\cos(2{\theta}_n^{(0)}(y))\,\dd y\lb{i2}\\
	&\,\,-d_n^2 \int_{-\Delta}^{\Delta} \left[\frac{S_n^2(y)}{2k}\right]\sin(2{\theta}_n^{(0)}(y))\dd y, \lb{i3}\\
	\wti X_n(k):=&\,\,\frac{d_n^2 |\hatt T_n(k)|^2}{8k^2}\cos{(4 \theta_n^{(0)}(0)+4 \phi_n(k))}-\frac{d_n^2 |\hatt S_n(k)|^2}{2}\cos{(4 \theta_n^{(0)}(0)+4 \psi_n(k))}\lb{extil}\\
	&\,\,+\frac{d_n^2}{2k}|\hatt S_n(k)\hatt T_n(k)| \sin(4 \theta_n^{(0)}(0)+2 \psi_n(k)+2 \phi_n(k)),
	\end{align}
	where	\begin{align}\lb{st}
		\begin{split}
			\phi_n(k):=\frac{\Ar(\hatt T_n(k))}{2},\qquad \psi_n(k):=\frac{\Ar(\hatt S_n(k))}{2},
		\end{split}
	\end{align}
	cf. \eqref{arz}, and the non-oscillatory term $\mathring X_n$ is given by
	\begin{align}
\overset{\circ}{X}_n(k):=&\frac{d_n}2 \left|\frac{\hatt T_n(k)}{2k}-\bfi \hatt S_n(k)\right|^2,\  n\geq 1.\lb{xhat}
\end{align}
\end{lemma}
\begin{proof}
	Integrating both sides of \eqref{ODE2} over the interval $[x_n-\Delta, x_n+\Delta]$ we get 
	\begin{align}\lb{delyn}
	\delta Y_n(k)& = \int_{-\Delta}^{\Delta}q_n(s)\sin(2{\theta}(x_n + s))-\sigma_n(s)\cos(2{\theta}(x_n + s))\,\dd s.
	\end{align}
Combining \eqref{eqlem33} and Taylor expansions for $\sin$, $\cos$ near $2\theta_n^{(0)}(s)$,
	\begin{align}
		\begin{split}\lb{2ndord}
				\sin(2\theta(x_n + s)) & \underset{n\rightarrow\infty}{=} \sin(2{\theta}^{(0)}_n(s)) + 2d_n\theta_n^{(1)}(s)\cos(2{\theta}^{(0)}_n(y)) + \cO(d_n^2)\,, \\
			\cos(2\theta(x_n + s)) &  \underset{n\rightarrow\infty}{=} \cos(2{\theta}^{(0)}_n(s)) -2d_n\theta_n^{(1)}(s)\sin(2{\theta}^{(0)}_n(s))  + \cO(d_n^2),
		\end{split}
	\end{align}
uniformly for $s\in[-\Delta, \Delta]$. Replacing the trigonometric terms in \eqref{delyn} by their second-order approximations \eqref{2ndord}, one infers
	\begin{align}
		&\delta Y_n(k) \underset{n\rightarrow\infty}{=}  \int_{-\Delta}^{\Delta}q_n(y)\sin(2{\theta}_n^{(0)}(y))-\sigma_n(y)\cos(2{\theta}_n^{(0)}(y))\,\dd y, \\
		& +2\int_{-\Delta}^{\Delta}d_n\theta_n^{(1)}(y)q_n(y)\cos(2{\theta}^{(0)}_n(y))+d_n\theta_n^{(1)}(y)\sigma_n(y)\sin(2{\theta}^{(0)}_n(y))\dd y +\cO(d_n^3),\lb{dynk}
	\end{align}
where the last cubic term was obtained by combining the linear \eqref{qnest} and the quadratic  \eqref{2ndord} asymptotic formulas.
In order to facilitate integration by parts in the subsequent argument, let us rewrite the terms in \eqref{dynk} containing $\theta_n^{(1)}$. First, use
the double angle formula to replace  $\sin^2$ term in \eqref{thetonenew}, 
	\begin{align}\lb{tn1}
	\theta^{(1)}_n(y)=\frac{1}{d_n}\int_{-\Delta}^y\Big[\sigma_n(s)\sin(2{\theta}_n^{(0)}(s))+q_n(s)\cos(2{\theta}_n^{(0)}(s))\Big]-q_n(s)\,\dd s.
\end{align}
Then, substitute this identity into the first term under the integral in \eqref{dynk} to get
	\begin{align}
		\begin{split}\lb{vsp1}
		&\int_{-\Delta}^{\Delta}d_n\theta_n^{(1)}(y)q_n(y)\cos(2{\theta}^{(0)}_n(y))\dd y=-\int_{-\Delta}^{\Delta}\Big[\int_{-\Delta}^yq_n(s)\dd s\Big]\ q_n(y)\cos(2{\theta}^{(0)}_n(y))\dd y\\
		&+\int_{-\Delta}^{\Delta}\Big[\int_{-\Delta}^y\sigma_n(s)\sin(2{\theta}_n^{(0)}(s))+q_n(s)\cos(2{\theta}_n^{(0)}(s))\,\dd s\Big] q_n(y)\cos(2{\theta}^{(0)}_n(y))\dd y
		\end{split}
\end{align}
and similarly, substitute \eqref{tn1} into the second  term under the same integral to get
\begin{align}
\begin{split}\lb{vsp2}
&\int_{-\Delta}^{\Delta}d_n\theta_n^{(1)}(y)\sin(2{\theta}^{(0)}_n(y))\sigma_n(y)\dd y
=-\int_{-\Delta}^{\Delta}\Big[\int_{-\Delta}^yq_n(s)\,\dd s\Big]\sigma_n(y)\sin(2{\theta}^{(0)}_n(y))\dd y\\
&+\int_{-\Delta}^{\Delta}\Big[\int_{-\Delta}^y\sigma_n(s)\sin(2{\theta}_n^{(0)}(s))+q_n(s)\cos(2{\theta}_n^{(0)}(s))\dd s\Big]\sigma_n(y)\sin(2{\theta}^{(0)}_n(y))\dd y.
\end{split}
\end{align}
Returning to $\delta Y_n(k)$, we plug \eqref{vsp1}, \eqref{vsp2} in \eqref{dynk}, use \eqref{vspint} with $f=\sigma_n(y)\sin(2{\theta}^{(0)}_n(y))$ and $g=q_n(s)\cos(2{\theta}_n^{(0)}(s))$; then, we obtain
\begin{align}
\begin{split}\lb{3line}
\delta Y_n(k) \underset{n\rightarrow\infty}{=} & \int_{-\Delta}^{\Delta}\left(q_n(y)-\int_{-\Delta}^yq_n(s)\,\dd s\,\sigma_n(y)\right)\sin(2{\theta}_n^{(0)}(y))\dd y\\
& -\int_{-\Delta}^{\Delta}\left(\int_{-\Delta}^yq_n(s)\dd s \ q_n(y)+\sigma_n(y)\right)\cos(2{\theta}_n^{(0)}(y))\,\dd y\\
&\left(\int_{-\Delta}^{\Delta}\sigma_n(y)\sin(2{\theta}^{(0)}_n(y))dy+\int_{-\Delta}^{\Delta}q_n(y)\cos(2{\theta}_n^{(0)}(y))dy\right)^2+ \cO(d_n^3).
\end{split}
\end{align}

Next, denote the quadratic term above by $L$ and note that $\theta_n^{(0)}(y)=\theta_n^{(0)}(0)+ky$; then,
\begin{align}
&L=\left(d_n\im e^{2\bfi \theta_n^{(0)}(0)}\hatt S_n(k)+\frac{d_n}{2k}\re e^{2\bfi \theta_n^{(0)}(0)} \hatt T_n(k)\right)^2\\
&=\left(d_n\im e^{2\bfi \theta_n^{(0)}(0)+2\bfi \psi_n(k)}|\hatt S_n(k)|+\frac{d_n}{2k}\re e^{2\bfi \theta_n^{(0)}(0)+2\bfi \phi_n(k)}|\hatt T_n(k)|\right)^2\\
&=\left(d_n \sin{(2 \theta_n^{(0)}(0)+2 \psi_n(k))}|\hatt S_n(k)|\right)^2+\left(\frac{d_n}{2k}\cos (2 \theta_n^{(0)}(0)+2 \phi_n(k))|\hatt T_n(k)|\right)^2\\
&\quad+\frac{d_n^2}{k}|\hatt S_n(k)\hatt T_n(k)| \sin{(2 \theta_n^{(0)}(0)+2 \phi_n(k))} \cos (2 \theta_n^{(0)}(0)+2 \phi_n(k))\\
&=\frac{d_n^2 |\hatt S_n(k)|^2}{2}-\frac{d_n^2 |\hatt S_n(k)|^2\cos{(4 \theta_n^{(0)}(0)+4 \psi_n(k))}}{2}\\
&\quad +\frac{d_n^2 |\hatt T_n(k)|^2}{8k^2}+\frac{d_n^2 |\hatt T_n(k)|^2\cos{(4 \theta_n^{(0)}(0)+4 \phi_n(k))}}{8k^2}\\
&\quad +\frac{d_n^2}{k}|\hatt S_n(k)\hatt T_n(k)| \sin{(2 \theta_n^{(0)}(0)+2 \psi_n(k))} \cos (2 \theta_n^{(0)}(0)+2 \phi_n(k)).
\end{align}
To conclude the derivation, we plug the above expression for $L$ in \eqref{3line}, expand $q_n, \sigma_n$ in terms of $d_n, S_n, T_n$ and combine the third order terms (with respect to $d_n$ as $n\rightarrow\infty$) with $\cO(d_n^3)$.
\end{proof}
\begin{remark}\lb{propap1}
Suppose that $f,g\in L^{1}(-\Delta, \Delta)$, then
	\begin{align}\lb{vspint}
	\begin{split}
	\frac12\left(\int_{-\Delta}^{\Delta}f(y)dy+\int_{-\Delta}^{\Delta}g(y)dy\right)^2&=\int_{-\Delta}^{\Delta}f(y)\int_{-\Delta}^{y}g(s)ds\,dy+\int_{-\Delta}^{\Delta}g(y)\int_{-\Delta}^{y}f(s)ds\,dy\\
	& + \int_{-\Delta}^{\Delta}f(y)\int_{-\Delta}^{y}f(s)ds\,dy+\int_{-\Delta}^{\Delta}g(y)\int_{-\Delta}^{y}g(s)ds\,dy.
	\end{split}
	\end{align}
	This identity follows from 
	\begin{equation}\lb{vspint1}
	\int_{-\Delta}^{\Delta}f(s)ds\int_{-\Delta}^{\Delta}g(y)dy=\int_{-\Delta}^{\Delta}f(y)\int_{-\Delta}^{y}g(s)ds\,dy+\int_{-\Delta}^{\Delta}g(y)\int_{-\Delta}^{y}f(s)ds\,dy,
	\end{equation}
	which is derived by changing the order of integration in the first integral on the right-hand side of \eqref{vspint1}.
\end{remark}

\begin{lemma}\lb{lem3.5} Recall $Y_n, \mathring X_n$ from Lemma \ref{lem3.4} and define 
\begin{equation}\lb{qns}
Q_n(k):=Y_n(k)-\sum_{m=1}^{n}\overset{\circ}{X}_m(k).
\end{equation}
Then for arbitrary non-negative $g\in C_0^{\infty}(0,\infty)$ with $\supp(g)\subset J$ for a  $(S, T)-$admissible interval $J$ we have
\begin{equation}
\lim\limits_{n\rightarrow\infty}\frac{\int\limits_{0}^{\infty}g(k)|Q_n(k)|dk}{\sum\limits_{m=1}^{n}d_m^2}=0.
\end{equation}
\end{lemma}
\begin{proof} Setting $Q_0=X_0=\wti X_0=\overset{\circ}{X}_0=0$, we note that
\begin{equation}\lb{mainlim}
Q_n(k)= \sum_{m=1}^n\delta Q_n(k) = \sum_{m=1}^n \left(X_m(k) + \wti {X}_m(k) + \cO(d_n^3)\right). 
\end{equation}
Define
\begin{align}
B_n:=\int\limits_{0}^{\infty}g(k)\Big|\sum\limits_{m=1}^{n}X_m(k)\Big|^2dk,\qquad \wti B_n:=\int\limits_{0}^{\infty}g(k)\Big|\sum\limits_{m=1}^{n}\wti X_m(k)\Big|^2dk;
\end{align}
then, by Cauchy--Schwarz inequality in $L^2(\bbR_+, dk)$, 
\begin{equation}
\int\limits_{0}^{\infty}g(k)|Q_n(k)|dk\leq\|\sqrt{g}\|_{L^2(\bbR_+)}\left(\sqrt{ B_n}+\sqrt{\wti B_n}\right)+ \cO(d_n^3).
\end{equation}
Following the proof of \cite[Theorem 1.6]{MR1628290}, we notice that by Stolz lemma (the discrete version of L'Hospital's rule), $\sum_{m=1}^{n}d_m^3/\sum_{m=1}^{n}d_m^2\rightarrow 0$ as $n\rightarrow\infty$; hence, in order to show \eqref{mainlim}, it suffices to prove
\begin{equation}\lb{bbtil}
\sqrt{{B}_n}\bigg/ \sum_{m=1}^nd_m^2 \to 0,\qquad \sqrt{\wti B_n}\bigg/ \sum_{m=1}^nd_m^2 \to 0,\qquad  n\rightarrow\infty. 
\end{equation}

To derive the first limit, recall $X_n$ from Lemma \ref{lem3.4} and denote the integral terms in \eqref{i1}, \eqref{i2}, \eqref{i3} by $U_n, V_n, Z_n$ respectively; thus, $X_n=d_n U_n-d_n V_n-d_n^2Z_n$. Put $M_{n-1}(k) := \sum_{m=1}^{n-1}X_m(k)$; then,
\begin{align}
	&B_n  \leq B_{n-1} +  \int g(k)|X_n(k)|^2dk\lb{lster}\\
	&\quad +2\left|\int g(k) M_{n-1}(k) d_nU_n(k)\,\dd k\right| + 2\left|\int g(k) M_{n-1}(k) d_nV_n(k)\,\dd k\right| \lb{gm1}\\
&\quad + 2\left|\int g(k) M_{n-1}(k) d_n^2Z_n(k)\,\dd k\right|  \lb{gm2}
\end{align}
Note that  $U_n, V_n, Z_n$ contain $\sin(2\theta^{(0)}_n(y))$, $\cos(2\theta^{(0)}_n(y))$ terms which we split in \eqref{gm1}, \eqref{gm2} using the triangle inequality. The resulting  terms are of the form
\begin{equation}\lb{cen1}
	\int\int_{-\Delta}^{\Delta} \gamma_n w_n(y) \Psi(k) M_{n-1}(k)u(2\theta^{(0)}_n(y))\dd y\,\dd k , 
\end{equation}
with $\gamma_n, w_n, \Psi, u$ as in \eqref{cen2}. As in the proof of Theorem \ref{sparseAltnewtheorem}(a), rewrite this quantity as 
\begin{equation}
	\int_{-\Delta}^{\Delta} \gamma_n w_n(y)\int \Psi(k) M_{n-1}(k) \frac{1}{2\frac{\partial \theta^{(0)}_n}{\partial k}(y)} \frac{\partial v}{\partial k}(2\theta^{(0)}_n(y))\,\dd k \,\dd y 
\end{equation}
where $v$ is either $\sin$ or $\cos$ so that $u=v'$. Next, integrate by parts with respect to $k$ and obtain three integrals, each corresponding to applying $\partial_k$ to one of the three functions in
\begin{equation}\lb{prod1}
	\Psi(k)\,\cdot\, M_{n-1}(k) \,\cdot\,\frac{1}{2\frac{\partial \theta^{(0)}_n}{\partial k}(y)}.
\end{equation}

{\it Case 1:} $\partial_k$ lands on the first term in \eqref{prod1}. In this case,
\begin{equation}\lb{en1}
\left|\int_{-\Delta}^{\Delta} \gamma_n w_n(y)\int \frac{\partial\Psi(k)}{\partial k} M_{n-1}(k) \frac{1}{2\frac{\partial \theta^{(0)}_n}{\partial k}(y)}  v(2\theta^{(0)}_n(y))\,\dd k \,\dd y \right|\leq \frac{Cd_n(n-1)}{x_n},
\end{equation}
 where we used \eqref{vsink} and $M_{n-1}(k)\leq C(n-1).$
 
{\it Case 2:} $\partial_k$ lands on the second term in \eqref{prod1}. Then,
\begin{align}\lb{en2}
\left|\int_{-\Delta}^{\Delta} \gamma_n w_n(y)\int \Psi(k) \frac{\partial M_{n-1}(k)}{\partial k} \frac{1}{2\frac{\partial \theta^{(0)}_n}{\partial k}(y)} v(2\theta^{(0)}_n(y))\,\dd k \,\dd y\right|\leq \frac{Cd_n}{x_n}\sum_{m=1}^{n-1}x_m\omega_m
\end{align}
where   and we used \eqref{vsink} and
 \begin{align}
	&  \frac{\partial M_{n-1}(k)}{\partial k}\leq \sum_{m=1}^{n-1}x_md_m.
\end{align}

{\it Case 3:} $\partial_k$ lands on the third term in \eqref{prod1}. We first replace $\Psi$ by $\frac{\Psi}{g}g$ and then estimate
\begin{align}
	\begin{split}\lb{en3}
	&\left|\int_{-\Delta}^{\Delta} \gamma_n w_n(y)\int g(k) M_{n-1}(k)\frac{\Psi(k)}{g(k)} \frac{1}{\left(\frac{\partial \theta^{(0)}_n}{\partial k}(y)\right)^2}\frac{\partial^2 \theta^{(0)}_n(y)}{\partial^2 k}v(2\theta^{(0)}_n(y))  \dd k \,\dd y\right|\\
	&\leq \frac{Cd_n}{x_n^2}\left(1+\sum_{m=1}^{n-1}d_mx_m^2\right)\left|\int g(k)  M_{n-1}(k)dk\right|\\
	&\leq \frac{Cd_n}{x_n^2}\left(1+\sum_{m=1}^{n-1}d_mx_m^2\right)\left(\int g(k)  |M_{n-1}(k)|^2dk\right)^{1/2}=: \alpha_n\sqrt{B_{n-1}}\,,
	\end{split}
\end{align}
where in the second to last inequality, we used \eqref{vsink} and 
 \begin{align}
	&  \frac{\partial^2 \theta^{(0)}_n(y)}{\partial^2 k}\leq C \left(1+\sum_{m=1}^{n-1}x_m\omega_m\right),\qquad\text{\ by \eqref{thetone}};
\end{align}
in the last inequality, we used the Cauchy--Schwarz inequality in $L^2(\bbR_+, dk)$ and denoted
\begin{equation}\lb{aln}
\alpha_n:=\frac{Cd_n}{x_n^2}\left(1+\sum_{m=1}^{n-1}d_mx_m^2\right).
\end{equation}

We are now ready to derive the first limit in \eqref{bbtil}: combine \eqref{gm1}, \eqref{gm2}, \eqref{en1}, \eqref{en2}, \eqref{en3}, and estimate the last term in \eqref{lster} from above by $Cd_n^2$, we have
\begin{align}\lb{inebns}
B_n\leq B_{n-1}+2\alpha_n\sqrt{B_{n-1}}+\beta_n,
\end{align}
where $\alpha_n$ is as in \eqref{aln} and 
\begin{equation}
\beta_n:=C\left(\frac{d_n(n-1)}{x_n}+\frac{d_n}{x_n}\sum_{m=1}^{n-1}x_md_m+d_n^2\right).
\end{equation}
Then, \eqref{inebns} together with \cite[Lemma 6.2]{MR1628290} yields
\begin{equation}\lb{bes}
\sqrt{B_n}\leq \sqrt{B_{0}}+\sum_{m=1}^{n}\alpha_m+\left(\sum_{m=1}^{n}\beta_m\right)^{1/2}.
\end{equation}
Consequently, the first limit in \eqref{bbtil} holds as asserted due to 
\begin{equation}\lb{alphabeta}
\sum_{m=1}^{n}\alpha_m\bigg/\sum_{m=1}^{n}d_m^2\rightarrow 0, \qquad \left(\sum_{m=1}^{n}\beta_m\right)^{1/2}\bigg/\sum_{m=1}^{n}d_m^2\rightarrow 0,\qquad n\rightarrow\infty,
\end{equation}
these two limits are discussed in Remark \ref{propap121} below.

Let us now derive the second limit in \eqref{bbtil}. First, we write $\wti X_n=d_n^2\wti U_n+d_n^2\wti V_n+d_n^2\wti Z_n$
where $\wti U_n, \wti V_n, \wti Z_n$ denote $k$-dependent functions in \eqref{extil}. Then, denoting $\wti M_{n-1}(k) := \sum_{m=1}^{n-1}\wti X_m(k)$, we obtain
 \begin{align}
 	\wti B_n  &\leq \wti B_{n-1} +  \int g(k)|X_n(k)|^2\dd k\lb{gm2new2}\\
 	&\quad+2\left|\int g(k) \wti M_{n-1}(k) d_n^2\wti U_n(k)\,\dd k\right| + 2\left|\int g(k) \wti M_{n-1}(k) d_n^2\wti V_n(k)\,\dd k\right| \lb{gm1new}\\
 	&\quad+ 2\left|\int g(k) \wti M_{n-1}(k) d_n^2\wti Z_n(k)\,\dd k\right|\lb{gm2new}.
 \end{align}
Note that  all three terms in \eqref{gm1new}, \eqref{gm2new} are of the form
\begin{equation}\lb{cen1new}
 \int \gamma_n \Psi(k)   \wti M_{n-1}(k)u(\mu_n(y))\,\dd k \,\dd y , 
\end{equation}
with $\Psi\in C_0^{\infty}(0,\infty)$, $\gamma_n:=d_n^2$ and
\begin{align}
&\mu_n(k)\in\{4 \theta_n^{(0)}(0)+2 \psi_n(k)+2 \phi_n(k),\,\, 4 \theta_n^{(0)}(0)+4 \psi_n(k),\,\, 4 \theta_n^{(0)}(0)+4 \phi_n(k)\}.
\end{align}
Using \eqref{thetatilde}, \eqref{thetone}(ii),  and Lemma \ref{lem312} (i) we get
\begin{align}
\frac{\partial \mu_n(k)}{\partial k} &> C x_n, \lb{dmu}\\
	\left|\frac{\partial^2\mu_n(k)}{\partial k^2} \right| &\leq C\left(1 + \sum_{m=1}^n d_m x_m^2 \right).\lb{ddmu}
\end{align} 
As in the first part of the proof, we proceed by rewriting \eqref{cen1new} in the form
\begin{equation}
	\int\gamma_n  \Psi(k) \wti M_{n-1}(k) \frac{1}{\frac{\partial \mu_n(k)}{\partial k}} \frac{\partial v}{\partial k}(\mu_n(k))\,\dd k,
\end{equation}
and integrating by parts with respect to $k$. This approach, as before, leads to three integrals, each corresponding to applying $\frac{\partial}{\partial k}$ to one of the three functions in
\begin{equation}\lb{prod1new}
	\Psi(k)\,\cdot\,\wti M_{n-1}(k) \,\cdot\,\frac{1}{\frac{\partial \mu_n(k)}{\partial k}}.
\end{equation}

{\it Case 1:} $\partial_k$ lands on the first term of \eqref{prod1new}. In this case,
\begin{align}
\Big|\int\gamma_n \frac{\partial \Psi(k)}{\partial k}  \wti M_{n-1}(k) \frac{1}{\frac{\partial \mu_n(k)}{\partial k}} v(\mu_n(k))\,\dd k\Big|\leq  \frac{C d_n(n-1)}{x_n} 
\end{align}
where we used $ \wti M_{n-1}(k)\leq C(n-1)$ and
\begin{align}\lb{vsink2}
	& \frac{1}{\frac{\partial \mu_n(k)}{\partial k}}\leq \frac{C}{x_n}\text{\ by\ }\eqref{dmu},\qquad \Psi \in C^{\infty}_0(0,\infty),\qquad |v(2\theta^{(0)}_n(y))|\leq 1. 
\end{align}

{\it Case 2:} $\partial_k$ lands on the second term of \eqref{prod1new}. In this case,
\begin{align}
\Big|\int\gamma_n \Psi(k)\frac{\partial \wti M_{n-1}(k)}{\partial k}  \frac{1}{\frac{\partial \mu_n(k)}{\partial k}} v(\mu_n(k))\,\dd k\Big|\leq\frac{Cd_n}{x_n}\sum_{m=1}^{n-1}x_md_m
\end{align}
where  we  used \eqref{vsink2} and
\begin{align}
&  \frac{\partial \wti M_{n-1}(k)}{\partial k}\leq \sum_{m=1}^{n-1}x_m d_m.
\end{align}

{\it Case 3:} $\partial_k$ lands on the third term of \eqref{prod1new}. We first replace $\Psi$ by $\frac{\Psi}{g}g$ and then estimate
\begin{align}
\begin{split}\lb{en3new}
&\left|\gamma_n \int g(k) \wti M_{n-1}(k)\frac{\Psi(k)}{g(k)} \frac{1}{\left(\frac{\partial \mu_n(k)}{\partial k}\right)^2}\frac{\partial^2 \mu_n(k)}{\partial k^2}v( \mu_n(k))  \dd k \,\dd y\right|\\
&\leq \frac{Cd_n}{x_n^2}\left(1+\sum_{m=1}^{n-1}d_mx_m^2\right)\left|\int g(k)  \wti M_{n-1}(k)dk\right|\\
&\leq \frac{C d_n}{x_n^2}\left(1+\sum_{m=1}^{n-1}d_mx_m^2\right)\left(\int g(k)  |\wti M_{n-1}(k)|^2dk\right)^{1/2}=: \alpha_n\sqrt{\wti B_{n-1}}
\end{split}
\end{align}
where in the second to last inequality, we used \eqref{ddmu} and \eqref{vsink2}; in the last inequality, we used the Cauchy--Schwarz inequality in $L^2(\bbR_+, dk)$ and the notation \eqref{aln}. 

Combining Cases 1-3, we get a version of \eqref{bes} with $B$ replaced by $\wti B$. As before, using \cite[Lemma 6.2]{MR1628290} and Proposition \ref{propap12}, we infer the second limit in \eqref{bbtil}.
\end{proof}
\begin{remark}\lb{propap121} Assuming the setting of Theorem \ref{sparseAltnewtheorem}(b). To prove the first limit in \eqref{alphabeta}, recall $\gamma$ from Remark \ref{propap12} and write
\begin{align}
\begin{split}
&\sum_{n=2}^{k} d_n\sum_{m=1}^{n-1}d_m\left(\frac{x_m}{x_n}\right)^2\leq  C\sum_{n=2}^{k} d_n\left(\frac{x_{n-1}}{x_n}\right)^2\sum_{m=1}^{n-1}d_m\left(\frac{x_m}{x_{n-1}}\right)^2\\
&\quad \leq C\sum_{n=2}^{k} d_n\left(\frac{x_{n-1}}{x_n}\right)^2\sum_{m=1}^{n-1}d_m\gamma^{-2|m-n-1|}\leq C\left(\sum_{n=2}^{k} d^2_n\left(\frac{x_{n-1}}{x_n}\right)^4\right)^{1/2} \left(\sum_{n=2}^{k} d^2_n\right)^{1/2}
\end{split}
\end{align}	
where in the last inequality, we used boundedness of the convolution operator, as in Remark \ref{propap12}. Note that 
\begin{equation}
\sum_{n=2}^{k} d^2_n\rightarrow\infty,\quad k\rightarrow\infty;\qquad\text{ and }\qquad\frac{x_{n-1}}{x_n}\rightarrow 0,\quad n\rightarrow \infty;
\end{equation}
thus, first limit in \eqref{alphabeta} holds. To prove the second limit in \eqref{alphabeta}, use \eqref{sumdns} to get
\begin{equation}
\sqrt{\beta_n}\leq C\sqrt{1+\sum_{m=1}^kd_m^2}=o\left(\sum_{n=2}^{k} d^2_n\right),\qquad  k\rightarrow\infty.
\end{equation}
\end{remark}

\begin{lemma}\lb{lem317}
Assume Hypothesis \ref{sparsepot}, fix a finite interval $[E_1, E_2]\subset (0,\infty)$ such that $[\sqrt{E_1}, \sqrt{E_2}]$ is $(S,T)-$admissible. Then  there exists a subsequence  $\{n_j\}_{j\geq 1}$ such that for Lebesgue almost every $k\in[\sqrt{E_1}, \sqrt{E_2}]$ one has
\begin{equation}\lb{rinf}
\lim\limits_{j\rightarrow\infty} R(x_{n_j}+\Delta, k)=\infty.
\end{equation}
\end{lemma}
\begin{proof}
Let $g\in C_0^{\infty}(0,\infty)$ be a strictly positive function with $[\sqrt{E_1}, \sqrt{E_2}]\subset \supp(g)\subset J$, for an $(S, T)-$admissible $J$, and $\int_{\bbR_+} g(k) dk=1$. Consider two sequences  $\xi_n:=Q_n$, cf. \eqref{qns}, 
$\zeta_n:=\sum_{m=1}^{n}\overset{\circ}{X}_m$, cf. \eqref{xhat}, of random variables in the probability space $(\Omega, \bbP):=((0,\infty), g(k)dk)$, and denote $\alpha_n:= \sum_{m=1}^{n}d_m^2$. Lemma \ref{lem3.5} and \eqref{newcondition} yield
\begin{align}
\lim\limits_{n\rightarrow\infty}\alpha_n^{-1}\bbE \xi_n=0\qquad \text{and}\qquad \zeta_n\geq C\alpha_n.
\end{align}
Then, by \cite[Lemma 6.1 (i), (ii')]{MR1628290}, there exists a subsequence $\{n_j\}_{j\geq 1}$ such that for almost every $k\in\cI$,
\begin{equation}
\lim\limits_{j\rightarrow\infty} (\xi_{n_j}(k)+\zeta_{n_j}(k))=\infty;
\end{equation}
pagthat is, $\lim\limits_{j\rightarrow\infty} Y_{n_j}(k)=\infty$ and therefore \eqref{rinf} holds as claimed. 
\end{proof}

\begin{proof}[Proof of Theorem \ref{sparseAltnewtheorem} (b)]\lb{pageb}
By Lemma \ref{lem2.4}, $\spec_{\ess}(H^{\alpha})=[0,\infty)$. Moreover, by Proposition \ref{prop3.9},  $H^{\alpha}$ has no positive eigenvalues.
	
	For every $(S, T)-$admissible interval $[\sqrt{E_1}, \sqrt{E_2}]$, by Lemma \ref{lem317} for some subsequence $\{n_j\}_{j=1}^{\infty}$ we have 
	\begin{equation}
		\lim\limits_{j\rightarrow\infty} R(x_{n_j}+\Delta, k)=\infty,\ a.e.\  k\in[\sqrt{E_1}, \sqrt{E_2}].
	\end{equation} Next, by Theorem \ref{cor2.11}, $\spec_{\ac} (H^{\alpha})\cap [E_1, E_2]=\emptyset$ and, since the union of all $(S, T)-$admissible intervals gives $\bbR_+$ up ot a discrete set, we conclude $\spec_{\ac} (H^{\alpha})=\emptyset$. Therefore, the spectrum of $H^{\alpha}$ is purely singular continuous on $(0,\infty)$. 
\end{proof}


\end{document}